\documentclass[11pt,reqno]{article}
\usepackage{amssymb, amsmath, amsthm, amsfonts, amscd, epsfig}
\usepackage{color}
\usepackage[english]{babel}
\usepackage{bm}
\usepackage{graphicx, epsfig}
\usepackage[normalem]{ulem}
\usepackage{setspace}
\usepackage{subcaption}
\newtheorem{theorem}{Theorem}[section]

\newtheorem{lemma}[theorem]{Lemma}

\newtheorem*{problem*}{Problem}

\newcommand{\Om}{\Omega}
\newcommand{\Bx}{\mathbf{x}}
\newcommand{\By}{\mathbf{y}}

\newcommand{\RR}{\mathbb{R}}
\newcommand{\CC}{\mathbb{C}}
\newcommand{\NN}{\mathbb{N}}
\newcommand{\ZZ}{\mathbb{Z}}
\newcommand{\Scal}{\mathcal{S}}

\newcommand{\p}{\partial}
\newcommand{\Ga}{\alpha}

\newcommand{\Gl}{\lambda}
\newcommand{\GG}{\Gamma}
\newcommand{\Gs}{\sigma}

\newcommand{\ds}{\displaystyle}
\newcommand{\eqnref}[1]{(\ref {#1})}

\newcommand{\Bu}{\mathbf{u}}

\newcommand{\beq}{\begin{equation}}
	\newcommand{\eeq}{\end{equation}}
\newcommand{\RN}[1]{%
	\textup{\uppercase\expandafter{\romannumeral#1}}%
}

\numberwithin{equation}{section}
\numberwithin{figure}{section}

\def\BK{{\bf K}}
\def\BS{{\bf S}}

\def\wtF{{\widetilde{F}}}
\def\wteta{{\widetilde{\eta}}}
\begin{document}

\title{Explicit analytic solution for the plane elastostatic problem with a rigid inclusion of arbitrary shape subject to arbitrary far-field loadings}
\author{Ornella Mattei$^1$ and Mikyoung Lim$^2$}
\date{\small{$^1$Department of Mathematics, San Francisco State University, CA 94132, USA\\
		$^2$Department of Mathematical Sciences, Korea Advanced Institute of Science and Technology, Daejeon 34141, Republic of Korea\\
		Email: mattei@sfsu.edu, mklim@kaist.ac.kr}}

\maketitle

\maketitle

\begin{abstract}
	\noindent
{We provide an analytical solution for the elastic fields in a two-dimensional unbounded isotropic body with a rigid inclusion.}
Our analysis is based on the boundary integral formulation of the elastostatic problem and geometric function theory. {Specifically, we use the coordinate system provided by the exterior conformal mapping of the inclusion to define a density basis functions on the boundary of the inclusion, and we use the Faber polynomials associated with the inclusion for a basis inside the inclusion.} The latter, which constitutes the main novelty of our approach, allows us to obtain an explicit series solution for the plane elastostatic problem for an inclusion of 
arbitrary shape in terms of the given arbitrary  far-field loading.  
\end{abstract}
%https://mathscinet.ams.org/mathscinet/msc/msc2020.html?t=74B05&s=30c55&btn=Search&ls=Ct
%74B05  	Classical linear elasticity
%30C55  	General theory of univalent and multivalent functions of one complex variable

\section{Introduction}

The so-called \textit{isolated inhomogeneity problem}, also known as the \textit{Eshelby problem}, consists in the determination of the elastic fields in an inclusion embedded in an unbounded medium, given some prescribed elastic fields in the far-field. Indeed, an inclusion with different material parameters from that of the background will induce some perturbation in the hosting medium which will depend on the shape of the inclusion as well as its material parameters. Such a problem has a long history (see, e.g., the review papers \cite{Kang:2009:CPS,Parnell:2016:EHM,Zhou:2013:ARR}), given its fundamental importance in material modelling, especially for its application to the determination of the effective properties of composite materials (see, for instance, the books by Buryachenko \cite{Buryachenko:2007:MHM}, Dvorak \cite{Dvorak:2013:MCM}, and Qu and Cherkaoui \cite{Qu:2006:FMS}). It was first addressed for an ellipsoidal inclusion by Poisson \cite{Poisson:1826:SMS} within the context of the Newtonian potential problem, then followed by Maxwell \cite{Maxwell:1873:TEM}, who provided some explicit formulas for the electric and magnetic fields inside the ellipsoid. 

With specific reference to the \textit{elastostatic case}, which is the topic of this paper, early results concerned the case of inclusions with simple shapes subject to uniform strains or tractions: just to name a few, see \cite{Goodier:1933:CSA,Southwell:1926:OCS} for spherical inclusions, \cite{Edwards:1951:SCA} for spheroids, and \cite{Robinson:1951:EEE,Sadowsky:1947:SCA,Sadowsky:1949:SCA} for ellipsoids. This problem is usually associated with Eshelby as his 1957 paper \cite{Eshelby:1957:DEF}, in which he proved that a homogeneous isotropic ellipsoidal inclusion embedded in an
unbounded homogeneous isotropic medium would experience uniform strains and
stresses when uniform strains or tractions were applied in the far-field, is one of the most cited papers in Applied Mechanics. Later, Eshelby \cite{Eshelby:1961:EII} conjectured that this occurs only for ellipsoidal inclusions, a fact that was proved independently by Kang and Miton \cite{Kang:2008:SPS} and Liu \cite{Liu:2008:SEC} for three-dimensional isotropic media subject to any uniform far-field loading, by Sendeckyi \cite{Sendeckyj:1970:EIP} for the two-dimensional case, and by Ru and Schiavone \cite{Ru:1996:EIA} for anti-plane elasticity.
Note that such a conjecture has not been proved yet for any uniform far-field loading in the case of anisotropic elastic media \cite{Asaro:1975:NUT,Kinoshita:1971:EFI,Walpole:1967:EFI}.

The extension of Eshelby's work \cite{Eshelby:1957:DEF} to \textit{inclusions of arbitrary shape} is not straightforward as it involves either the computation of complicated Green's functions (e.g., \cite{Chiu:1977:OTS,Nozaki:1997:EFP}), or solutions only in algorithmic closed form \cite{Rodin:1996:EIP}. 
In practical applications, on the other hand, such as in metallurgy, in which the goal is to model perturbations of elastic fields due to precipitates, twinnings and martensitic transformations, the inclusions have a more complex shape. 
To overcome such a drawback, some authors applied the theory of conformal mapping to the Eshelby problem in plane elastostatics by relying on the complex formulation introduced by Muskhelishvili \cite{Muskhelishvili:1963:SBP} (see also \cite{England:1971:CVM,Lu:1995:CVM}). 
This method can be applied to elastic inclusions of arbitrary shape, but in general, it does not provide an explicit solution (see, for instance, \cite{List:1966:TDE,Sendeckyj:1970:EIP}), and in those cases in which an explicit solution can be provided, the far-field loading is considered to be uniform \cite{Movchan1997ThePM,Ru:1999:ASE,Zou:2010:EPN}. The determination of \textit{explicit formulas} for the elastic fields in an inclusion of arbitrary shape subject to an arbitrary far-field loading is still, to the best of the authors' knowledge, an open problem.

We remark that analytical and numerical methods to compute the elastic tensor (often called the Eshelby tensor field) for inclusions of various shapes have been developed \cite{Chen:2015:NEE,Gao:2010:SGS,Huang:2009:EEE,Lee:2016:EEF,Zou:2010:EPN}. {
In particular, in \cite{Zou:2010:EPN}, the inclusion problem to determine the strain field of an infinite homogeneous medium induced by a uniform eigenstrain field in an inclusion was 	extensively studied. 
Eshelby's tensor field relates the resulting strain field to the uniform eigenstrain, and it has an expression as an integral on the boundary of the inclusion in terms of the derivatives of the Green's function of the elastic body. For inclusions of various shapes, given the corresponding conformal mappings, Eshelby's tensor field is explicitly expressed in \cite{Zou:2010:EPN}. 
In the present paper, we consider the inclusion problem to find the deformation due to an inclusion for a given arbitrary loading, when the solution is expressed as the single-layer potential having  the Green's function as kernel.}
% {\color{red} 
%In particular, in \cite{Zou:2010:EPN}, the Eshelby problem, which is to determine the strain field of an infinite homogeneous medium induce by a uniform eigenstrain field on an inclusion, was intensively studied. 
%Eshelby's tensor field relates the resulting strain field to the uniform eigenstrain, and it has an expression as an integral on the boundary of the inclusion in terms of the derivatives of the Green's function of the elastic body. For inclusions of various shapes given the conformal mapping, Eshelby's tensor field is explicitly expressed in \cite{Zou:2010:EPN}. 
%In the present paper, we consider the inclusion problem to find the deformation due to an inclusion for a given arbitrary loading, where the solution has an expression as the single-layer potential whose kernel function is  the Green's function.}

For the conductivity problem (or antiplane elasticity), Jung and Lim \cite{Jung:2018:SSM} found an explicit solution for an inclusion of arbitrary shape based on the layer potential technique and geometric function theory: the solution inside the inclusion is expanded into a series of Faber polynomials (see \cite{Duren:1983:UF,Faber:1903:UPE} for the definition and properties), whereas the solution in the surrounding medium is expanded into a series of harmonic functions expressed in terms of the coordinates of the exterior conformal mapping, supposed to be known.  The layer potential operators and the Neumann--Poincar{\'e} operator admit expansions into series of the basis functions so that one can find a series solution to the transmission problem by using the layer potential technique; we refer the reader to see, for instance, the comprehensive books \cite{Ammari:2013:MSM:book,Ammari:2004:RSI} for  the layer potential formulation for the conductivity transmission problem and its applications.
Moreover, the two sets of interior and exterior basis functions have explicit relations on the boundary of the inclusion so that the interior and exterior values of the solution can be matched by using the transmission condition on the boundary of the inclusion (see also \cite{Choi:2018:GME:preprint}).
This geometric series solution method was successfully applied to the study of inclusion problems in anti-plane elasticity and the spectral property of the Neumann--Poincar{\'e} operator \cite{Choi:2020:ASR:preprint,Kim:2018:EEC:preprint,Jung:2020:DEE}.
Analogous results for the elastic problem have not been found yet. 

In the present paper, we extend the series solution approach for the conductivity transmission problem in \cite{Jung:2018:SSM} to the elastostatic case. In particular, we provide an explicit analytic series solution for the hard inclusion problem in plane elastostatics.

The remainder of this paper is organized as follows. Section 2 mainly describes the complex formulation for the plane elastostatic transmission problem.
Section 3 is devoted to the series solution for the rigid inclusion problem.
We provide solutions for inclusions of various orders in Section 4.
The paper ends with some concluding remarks in Section 5.

\section{Preliminary}

\subsection{Formulation of the problem}

Consider an unbounded homogeneous isotropic medium in $\mathbb{R}^2$ with an embedded simply connected homogeneous isotropic inclusion $\Omega$ of arbitrary shape.  Let $\lambda$ and $\mu$ be the Lam\'e constants of the system, $\lambda$ being the bulk modulus and $\mu$ the shear modulus. 
Suppose that the displacement field, $\mathbf{u}(\mathbf{x})$, is assigned in the far-field in a quasi-static manner, that is $\mathbf{u}(\Bx)=\mathbf{u}_0(\Bx) + O(|\Bx|^{-1})$  as $|\mathbf{x}|\to \infty$, with $\mathbf{u}_0(\Bx)$ the far-field displacement.

If there were no inclusion, the displacement field $\mathbf{u}(\mathbf{x})$ would be exactly $\mathbf{u}_0(\mathbf{x})$, whereas the strain field $\boldsymbol{\varepsilon}(\mathbf{x})$  and the stress field $\boldsymbol{\sigma}(\mathbf{x})$ would be determined uniquely by the following well-known formulas (assuming  the displacements to be small):
\begin{equation}\label{eqn:stress_strain}
\boldsymbol{\varepsilon}=\frac{1}{2}\left(\nabla\mathbf{u}+\nabla^{\footnotesize\mathrm{T}}\mathbf{u}\right)\quad\mbox{and}\quad \boldsymbol{\sigma}=\mathbf{C}\hskip 0.01cm\boldsymbol{\varepsilon}=\lambda\,\text{tr}\boldsymbol{\varepsilon} +2\mu\,\boldsymbol{\varepsilon},
\end{equation}
where tr stands for the trace. If we assume there are no body forces (so that the stress field is divergence-free), then $\mathbf{u}_0(\mathbf{x})$ satisfies the equation $\mathcal{L}_{\lambda,\mu} \mathbf{u}_0= \mathbf{0} $ in $\mathbb{R}^2$, $\mathcal{L}_{\lambda,\mu}$ being the following differential operator
$$
\mathcal{L}_{\lambda,\mu} \mathbf{u}:= \nabla\cdot \mathbf{C}\hskip 0.01cm\boldsymbol{\varepsilon}(\mathbf{x})=\mu \Delta \mathbf{u}+(\lambda+\mu)\nabla \nabla \cdot \mathbf{u}.
$$
Note that $\mathcal{L}_{\lambda,\mu}$ is elliptic under the strong convexity assumption for which 
$\mu>0$ and $\lambda+\mu>0$ (see, e.g., \cite{Kupradze:1965:PMT}). 

However, due to the presence of the inclusion $\Om$, the displacement field in the medium is not $\mathbf{u}_0(\mathbf{x})$. Indeed, under the assumptions that the inclusion is rigid (so that the displacement field on $\p \Omega$ is a rigid displacement), and that there are no body forces, $\Bu(\Bx)$ turns out to be the solution of the system
\begin{equation} \label{elas_eqn}
\ \left \{
\begin{array} {ll}
\ds \mathcal{L}_{\lambda,\mu} \mathbf{u}= \mathbf{0} \quad &\mbox{ in }  \mathbb{R}^2 \setminus \overline{\Om},\\[2mm]
\ds \mathbf{u}\big|^+ =\sum_{j=1}^3 c_j\mathbf{R}_j(\Bx) \quad &\mbox{ on } \p \Om,\\[2mm]
\ds \mathbf{u}(\Bx)-\mathbf{u}_0(\Bx) = O(|\Bx|^{-1}) \quad& \mbox{ as } |\Bx| \rightarrow \infty,
\end{array}
\right.
\end{equation}
where the superscript $+$ denotes the limit from outside $\Om$, and $\{\mathbf{R}_1,\mathbf{R}_2,\mathbf{R}_3\}$ is a basis of the space of rigid displacements, say
\begin{equation}
\mathbf{R}_1=\left[\begin{array}{c}
1\\
0
\end{array}\right]\,,\quad \mathbf{R}_2=\left[\begin{array}{c}
0\\
1
\end{array}\right]\,,\quad \mathbf{R}_3=\left[\begin{array}{c}
x_2\\
-x_1
\end{array}\right].
\end{equation}

The conormal derivative from outside on $\p \Om$, that represents the traction forces on the exterior of $\p \Om$, is defined to be 
$$\p_\nu \mathbf{u} :=\lambda\left(\nabla\cdot \mathbf{u}\right)\mathbf{n} + \mu \left(\nabla\mathbf{u}+\nabla^{\footnotesize\mathrm{T}}\mathbf{u}\right)\mathbf{n},$$
$\mathbf{n}$ being the unit outward normal vector to $\p\Om$. 
The real coefficients $c_j$, $j=1,2,3$, in \eqref{elas_eqn} are determined by the following equilibrium condition on $\p\Om$:
\begin{equation}\label{equilibrium_boundary}
\int_{\p\Om}\p_{\nu}\mathbf{u}\big|^+\cdot\mathbf{R}_jd\sigma=0,\quad j=1,2,3.
\end{equation}

\subsection{Layer potential technique}

A classical way of solving the transmission problem \eqref{elas_eqn} is the layer potential technique (see, e.g., \cite{Kupradze:1965:PMT}), which is based on the ansatz that the displacement field in the medium is the superposition of the far-field loading $\Bu_0(\Bx)$ (the field that would be in the medium if there were no inclusion), and a perturbation field $\mathbf{S}_{\p \Omega}[\boldsymbol{\varphi}](\Bx)$ (the term which takes into account the effect of the inclusion):
\begin{equation}\label{relation_u_u0_S}
\mathbf{u}(\Bx)=\mathbf{u}_0(\Bx)+\mathbf{S}_{\p \Omega}[\boldsymbol{\varphi}](\Bx),
\end{equation}
for some density function $\mathbf{\boldsymbol{\varphi}}=(\varphi_1,\ \varphi_2)^T$ satisfying the equilibrium condition \eqref{equilibrium_boundary}, i.e.
\beq\label{cond:phi}
\int_{\p\Om}{\boldsymbol{\varphi}}\cdot\mathbf{R}_jd\sigma=0,\quad j=1,2,3.
\eeq
Indeed, $\boldsymbol{\varphi}(\Bx)$ is given by ${\boldsymbol{\varphi}}=\p_\nu \mathbf{{u}}\big|^+$ on $\p\Om$ (see, e.g., \cite[Appendix A.2]{Kang:2019:QCS}).

The perturbation field $\mathbf{S}_{\p \Omega}[\boldsymbol{\varphi}](\Bx)$ is called \textit{the single-layer potential}  of the density function $\boldsymbol{\varphi}(\Bx)$ on $\p\Omega$ associated with the Lam\'{e} system, and it is defined as
\begin{equation}
\label{single_layer}
\BS_{\p \Om} [\boldsymbol{\varphi}] (\Bx)  := \int_{\p \Om} \boldsymbol{\Gamma}(\Bx - \By) [\boldsymbol{\varphi}] (\By) \,\mathrm{d}\Gs (\By) \quad \mbox{for }\Bx \in \RR^2,
\end{equation}
where $\boldsymbol{\Gamma}=(\GG_{ij})_{i,j=1}^2$ is the Kelvin matrix of the fundamental solution to the Lam\'{e} system in $\RR^2$ when there is no inclusion, namely,
\begin{equation}\label{eqn:Kelvin_matrix}
\GG_{ij}(\Bx) = \frac{\Ga_1}{2\pi} \delta_{ij} \ln |\Bx| - \frac{\Ga_2}{2\pi} \frac{x_i x_j}{|\Bx|^2},
\end{equation}
$\delta_{ij}$ being Kronecker's delta, and
\begin{equation}\label{alpha1_2}
\Ga_1 = \frac{1}{2}\left( \frac{1}{\mu} + \frac{1}{2 \mu + \Gl}\right) \quad \mbox{and} \quad \Ga_2 = \frac{1}{2} \left(\frac{1}{\mu} - \frac{1}{2 \mu + \Gl}\right).
\end{equation}
Note that, due to \eqref{relation_u_u0_S}, the problem of finding the displacement field $\mathbf{u}(\Bx)$ satisfying \eqref{elas_eqn} is equivalent to finding the single-layer potential $\mathbf{S}_{\p \Omega}[\boldsymbol{\varphi}](\Bx)$ \eqref{single_layer} by finding the density function $\boldsymbol{\varphi}(\Bx)$ satisfying \eqref{cond:phi}.
The density function $\boldsymbol{\varphi}(\Bx)$ is associated with the inversion of the operator $\displaystyle -\frac{1}{2}\mathbf{I}+\mathbf{K}^*_{\p\Omega}$, where $\mathbf{K}^*_{\p\Omega}$ is the so-called elastostatic Neumann--Poincar\'e operator:
\begin{equation*}
\BK^*_{\p \Om} [\boldsymbol{\varphi}] (\Bx) := \mathrm{p.v.} \int_{\p \Om} \p_{\nu_{\Bx}}\GG (\Bx - \By) \boldsymbol{\varphi}(\By)\, \mathrm{d}\Gs (\By) \quad \mbox{a.e.} \quad \Bx \in \p \Om.
\end{equation*}
The symbol  p.v. stands for the Cauchy principal value, $\p_{\nu_{\Bx}}\GG (\Bx - \By)$ denotes the conormal derivative of the Kelvin matrix with respect to the $\Bx$-variable, defined by
\begin{equation*}
\p_{\nu_{\Bx}}\GG (\Bx - \By)\mathbf{b}=\p_{\nu_{\Bx}}\left(\GG (\Bx - \By)\mathbf{b}\right)
\end{equation*}
for any constant vector $\mathbf{b}$.

\subsection{Complex formulation}\label{subsec:complex}

For plane elastostatics, one could identify  the coordinate vector $\Bx=(x_1,\, x_2)^\mathrm{T}\in\RR^2$ with the complex variable $z=x_1+\mathrm{i}x_2\in\CC$. The vector-valued displacement field $\mathbf{u}=(u_1,\, u_2)^\mathrm{T}$ then can be expressed, upon complexification, as the complex function $u(z)= u_1+\mathrm{i}u_2$. 
Analogously, the density function $\mathbf{\boldsymbol{\varphi}}=(\varphi_1,\, \varphi_2)^T$ and the single-layer potential $\mathbf{S}_{\p \Omega}[\boldsymbol{\varphi}](\Bx)$ can be written as $\varphi(z)=\varphi_1+\mathrm{i}\varphi_2$ and $S[\varphi](z)=(\BS_{\p\Om}[\boldsymbol{\varphi}])_1+\mathrm{i}(\BS_{\p\Om}[\boldsymbol{\varphi}])_2$, respectively.

Following \cite{Muskhelishvili:1963:SBP}, the equation $\mathcal{L}_{\lambda,\mu}\mathbf{u}=\mathbf{0}$ in $\CC\setminus\overline{\Om}$, in which the two unknowns are $u_1(\Bx)$ and $u_2(\Bx)$, can then be written as an equation in terms of two complex functions $\phi(z)$ and $\psi(z)$, which are analytic in $\CC\setminus\overline{\Om}$:
\begin{equation}\label{rep_u_kappa}
2\mu \hskip .01cm u(z)=\kappa \phi(z)-z\overline{\phi'(z)}-\overline{\psi(z)},\quad \kappa=\frac{\lambda+3\mu}{\lambda+\mu},
\end{equation}
where the bar denotes complex conjugation. Similarly, the background solution $u_0(z)=(\mathbf{u}_0)_1+\mathrm{i}(\mathbf{u}_0)_2$ admits the complex representation (dividing by the constant $\mu$ to simplify the formula)
\begin{equation}\label{eqnd:complex_repr_u0}
2u_0(z)=\kappa h(z)-z\overline{h'(z)}-\overline{l(z)}\quad\mbox{in }\Om\ (\mbox{or in }\CC\setminus\overline{\Om}),
\end{equation}
where $h(z)$ and $l(z)$ are analytic functions in $\Om$ (or in $\CC\setminus\overline{\Om}$). 
Indeed, $\mathbf{u}_0(\mathbf{x})$ satisfies $\mathcal{L}_{\lambda,\mu} \mathbf{u}_0= \mathbf{0} $ in $\mathbb{R}^2$, so that the representation \eqnref{eqnd:complex_repr_u0} holds in the whole complex plane.

As the single-layer potential $\mathbf{S}_{\p \Omega}[\boldsymbol{\varphi}](\Bx)$ satisfies the Lam\'{e} system as well, the following holds 
\begin{equation}\label{complex_form_S}
2S[\varphi](z)=\kappa f(z)-z\overline{f'(z)}-\overline{g(z)} \quad\mbox{in }\Om\ (\mbox{or in }\CC\setminus\overline{\Om})
\end{equation}
for some complex analytic functions $f(z)$ and $g(z)$, which can be expressed as
\begin{align}
\label{f}f(z)&=f[\varphi](z)={\alpha_2}\mathcal{L}[\varphi](z),\\
\label{g}g(z)&=g[\varphi](z)=-{\alpha_1}\mathcal{L}[\overline{\varphi}](z)
-{\alpha_2}\mathcal{C}[\overline{\zeta}\varphi](z)
\end{align}
with the complex integral operators $\mathcal{L}$ and $\mathcal{C}$ given by
\begin{align}
\label{mathcal_L}\mathcal{L}[\psi](z)&:=\frac{1}{2\pi}\int_{\p \Om}\log(z-\zeta)\psi(\zeta)\,\mathrm{d}\sigma(\zeta),\\
\label{mathcal_C}\mathcal{C}[\psi](z):&=\mathcal{L}[\psi]'(z)=\frac{1}{2\pi}\int_{\p \Om}\frac{\psi(\zeta)}{z-\zeta}\,\mathrm{d}\sigma(\zeta)
\end{align}
for any complex function $\psi(z)$ (see \cite{Ando:2018:SPN} for the derivation).
Indeed, we can express the single-layer potential $S[\varphi](z)$ in a simpler form by using the real logarithmic function (note that, by \eqnref{alpha1_2} and \eqnref{rep_u_kappa}, $\kappa\alpha_2=\alpha_1$):
\beq\label{singlelayer:simple}
2S[\varphi](z)=2{\alpha_1}L[\varphi](z)-\alpha_2 z\, \overline{\mathcal{C}[\varphi](z)}+\alpha_2\overline{\mathcal{C}[\overline{\zeta}\varphi](z)}\quad\mbox{in }\CC
\eeq
with 
\beq\label{singlelayer:real}
L[\varphi](z):=\frac{1}{2\pi}\int_{\p\Om}\ln |z-\zeta|\varphi(\zeta)\, d\sigma(\zeta).
\eeq

The rigid inclusion transmission condition on $\p \Om$, see \eqref{elas_eqn}, and the single-layer potential ansatz \eqref{relation_u_u0_S} imply that, on the boundary, 
\begin{equation}\label{trans_cond_S_u0}
S[\varphi](z)=-{u}_0(z)+c_1+\mathrm{i}c_2- \mathrm{i}c_3 z.
\end{equation}

In Section \ref{sec:series_solution}, we will develop a series solution method for the transmission problem \eqref{elas_eqn} by using the theory of \textit{conformal mapping}. Therefore, we review some related results in complex analysis in the following subsection.

\subsection{Exterior conformal mapping and associated density basis functions}\label{subsec:exterior}

From the Riemann mapping theorem, there uniquely exist a real number $\gamma>0$ and a complex function $\Psi(w)$ that conformally maps the region $\mathcal{R}=\{w\in\CC:|w|>\gamma\}$ onto $\mathbb{C}\setminus\overline{\Omega}$ and satisfies $\Psi(\infty)=\infty$ and $\Psi'(\infty)=1$. {Here, the quantity $\gamma$ is called the conformal radius of $\Om$, which coincides with the logarithmic capacity of $\overline{\Om}$}, and $\Psi$ is the \textit{exterior conformal mapping} associated with $\Om$. 
The function $\Psi(w)$ admits the following Laurent series expansion:
\beq\label{eqn:extmapping} 
\Psi(w)=w+{a}_0+\frac{{a}_1}{w}+\frac{{a}_2}{w^2}+\cdots=w+\sum_{k=0}^{\infty}a_kw^{-k}
\eeq 
for some complex coefficients $a_n$, see \cite[Chapter 1.2]{Pommerenke:1992:BBC:book} for the derivation.
Being a conformal mapping, $\Psi(w)$ preserves the angle between two intersecting curves and, hence, it can define an orthogonal curvilinear coordinate system in $\mathbb{C}\setminus\overline{\Om}$ in a simple way. 
Indeed, we can express $w\in\mathcal{R}$ in the modified polar coordinates $(\rho,\theta)\in(\rho_0,\infty)\times[0,2\pi)$ with $\rho_0=\ln \gamma$ as 
\begin{equation}\label{w}
w=e^{\rho+\mathrm{i}\theta}.
\end{equation}
Consequently, $(\rho,\theta)$ provides an orthogonal coordinate pair for $z=\Psi(w)\in\CC\setminus\overline{\Om}$ via the relation $$z=\Psi(\rho,\theta):=\Psi(e^{\rho+\mathrm{i}\theta}).$$

From the Caratheodory extension theorem \cite{Caratheodory:1913:GBR}, $\Psi(\rho,\theta)$ admits the continuous extension to the boundary of the domain. We assume that the boundary $\p\Om$ is $C^{1,\alpha}$ so that, by the Kellogg--Warschawski theorem \cite{Pommerenke:1992:BBC:book}, $\Psi'(\rho,\theta)$ can be continuously extended to the boundary. In particular, the map $\Psi(\rho,\theta)$ is $C^1$ on $[\rho_0,\infty)\times[0,2\pi)$. {Therefore, we can use the polar coordinate system $(\rho,\theta)$ associated with \eqref{eqn:extmapping} to define a density basis, that is, a basis for functions defined on $\p\Om$ ($\rho=\rho_0$)}, as follows
\begin{equation}\label{exterior_basis}
\varphi_m(z)=\frac{\mathrm{e}^{\mathrm{i}m\theta}}{h},\quad \varphi_{-m}(z)=\frac{\mathrm{e}^{-\mathrm{i}m\theta}}{h}
\end{equation}
for each $m\in\NN$, where $h$ is the scale factor
% with respect to $\rho$ that equals the scale factor with respect to $\theta$:
\begin{equation}\notag
h(\rho, \theta) = \left|\frac{\partial \Psi} {\partial \rho}\right| = \left|\frac{\partial \Psi} {\partial \theta}\right|.
\end{equation}
{Later, we will use the basis \eqnref{exterior_basis} to expand the density function $\varphi(z)$ of the single-layer potential \eqref{single_layer}.}
Notice that the length element on $\p\Om$ is then given by
\beq\label{dsigma}
\mathrm{d}\sigma(\zeta)=h(\rho_0,\theta)\mathrm{d}\theta\quad\mbox{for }\zeta=\Psi(\rho_0,\theta).\eeq
\subsection{Faber polynomials}

In view of the exterior conformal mapping \eqref{eqn:extmapping}, the expansion
\begin{equation}\label{eqn:Fabergenerating}
\frac{w\Psi'(w)}{\Psi(w)-z}=\sum_{m=0}^{\infty}F_m(z)w^{-m}
\end{equation}
is valid for sufficiently large $|w|$, where 
the function $F_m(z)$ is an $m$-th order monic polynomial, called the $m$-\textit{th Faber polynomial} associated with $\Psi(w)$ (or $\Omega$), that is uniquely determined by the coefficients of $\Psi(w)$ in \eqref{eqn:extmapping}. 
As an example, the first three polynomials are
$$F_0(z)=1,\quad F_1(z)=z-a_0,\quad F_2(z)=z^2-2a_0 z+(a_0^2-2a_1).$$
In general, by comparing the $w^{-m}$ terms in \eqref{eqn:Fabergenerating} (after multiplying both sides by $\Psi(w)-z$) one observes the following recursion relation\begin{equation}\label{eqn:Faberrecursion}
-ma_m=F_{m+1}(z)+\sum_{s=0}^{m}a_{s}F_{m-s}(z)-zF_m(z) \qquad\mbox{for each } m= 0,1,2,\dots.
\end{equation}
The concept of Faber polynomials was first introduced by G. Faber in \cite{Faber:1903:UPE} and has been one of the essential elements in geometric function theory (see, e.g., \cite{Duren:1983:UF}).

The Faber polynomials satisfy the convenient property that $F_m(\Psi(w))$ has only one positive term: $w^m$. In other words,
\begin{equation}\label{eqn:Faberdefinition}
F_m(\Psi(w))%=w^m+\frac{c_{m,1}}{w}+\frac{c_{m,2}}{w^2}+\cdots=
=w^m+\sum_{k=1}^{\infty}c_{m,k}{w^{-k}}.
\end{equation}
The coefficients $c_{m,k}$ are called the \textit{Grunsky coefficients}, and they satisfy the following identity: 
\beq\label{Grun:iden}
k c_{m,k}=m c_{k,m}\quad\mbox{for all }m,k\in\NN.
\eeq
Furthermore, they satisfy the so-called {\it strong Grunsky inequalities} \cite{Grunsky:1939:KSA} (see also \cite{Duren:1983:UF}): let $N$ be a positive integer and $\lambda_1,\lambda_2,\dots,\lambda_N$ be complex numbers that are not all zero, then we have
\beq\label{inequal:strong}
\sum_{k=1}^\infty k\left|\sum_{n=1}^N\frac{c_{n,k}}{\gamma^{n+k}}\lambda_n\right|^2\leq\sum_{n=1}^N n\left|\lambda_n\right|^2.
\eeq
The strict inequality holds unless $\Om$ has measure zero. 
As a corollary, the following relation,  the so-called {\it weak Grunsky inequality}, holds:
\beq\label{inequal:week}
\left|\sum_{n=1}^N\sum_{k=1}^N k\frac{c_{n,k}}{\gamma^{n+k}}\lambda_n \lambda_k\right|\leq\sum_{n=1}^N n\left|\lambda_n\right|^2.
\eeq
Plugging $\lambda_n=\frac{1}{\sqrt{m}}\delta_{nm}$ into \eqnref{inequal:strong}, we have 
$$\sum_{k=1}^\infty \left|\sqrt{\frac{k}{m}}\frac{c_{m,k}}{\gamma^{m+k}}\right|^2\leq 1.$$
In particular, $\left|\frac{c_{m,m}}{\gamma^{2m}}\right|\leq 1$. Now, letting $\lambda_n=\frac{1}{m}\delta_{nm}+\frac{1}{s}\delta_{ns}$ for $m\neq s$, \eqnref{inequal:week} leads us to 
$$\left|\frac{c_{m,m}}{m\gamma^{2m}}+\frac{c_{m,s}}{m\gamma^{m+s}}+\frac{c_{s,m}}{s\gamma^{s+m}}+\frac{c_{s,s}}{s\gamma^{2s}}\right|\leq\frac{1}{m}+\frac{1}{s}\leq 2$$
and, by using \eqnref{Grun:iden}, 
\beq\label{c:bound}
\left|c_{m,s}\right|\leq 2 m \gamma^{m+s}.
\eeq
By applying this bound to \eqnref{eqn:Faberdefinition}, it follows that
\beq\label{Faber:bound}
\left|F_m(\Psi(w)\right|\leq |w|^m+2m\gamma^m\frac{\gamma}{|w|-\gamma}.
\eeq

We can derive more formulas for Faber polynomials starting from the generating relation \eqnref{eqn:Fabergenerating}.
For instance, by integrating \eqnref{eqn:Fabergenerating} with respect to $w$, we have (see, e.g., \cite[Chapter 5]{Duren:1983:UF})
$$
\log\left(\frac{w}{\Psi({w})-{z}}\right)=\sum_{m=1}^\infty \frac{1}{m}F_m({z})w^{-m}.
$$
Differentiation of this relation with respect to $z$ leads to
\beq\label{fund:deri:z}
\frac{1}{\Psi(w)-z}=\sum_{m=1}^\infty\frac{1}{m}F_m'(z)w^{-m}.
\eeq
Note that the complex function $1/({\Psi(w)-z})$ is analytic with respect to $w$ in $\{w:|w|>\gamma\}$ and decays at infinity, while the right-hand side is its Laurent series. Hence, for a fixed $z\in\Om$, \eqnref{fund:deri:z} is uniformly and absolutely convergent for $|w|>\gamma_1$ with $\gamma_1>\gamma$.

Note that, thanks to the properties of Faber polynomials, any complex function $v(z)$ analytic in ${\Omega_R}$, $R>\gamma$, admits the expansion
\beq\label{exp_anal_function_Faber}
v(z)=\sum_{m=0}^{\infty}d_m F_m(z)\quad\mbox{in }\overline{\Om}
\eeq
with the coefficients $d_m$ given by (from the Cauchy integral formula and \eqnref{eqn:Fabergenerating})
\beq\label{fabercoeff}
d_m = \frac{1}{2\pi \mathrm{i}}\int_{|w|=r}\frac{v(\Psi(w))}{w^{m+1}}\, dw,\quad \gamma<r<R.
\eeq
In other words, Faber polynomials form the \textit{interior basis}: they can be used as an expansion basis on $\overline{\Om}$ for complex functions that are analytic on a domain containing $\overline{\Om}$.

\section{Series solution for the rigid inclusion problem}\label{sec:series_solution}
The goal of this section is to develop a series solution method for the transmission problem \eqref{elas_eqn}, in order to provide the displacement field in the exterior of the inclusion. Specifically, we will expand the far-field loading in terms of the interior basis and the single-layer potential in terms of the density basis \eqref{exterior_basis}. By using the transmission condition \eqref{trans_cond_S_u0} and the properties of Faber polynomials, we will find an explicit expression for the elastic fields in the exterior of the inclusion in terms of the coordinate $w$ \eqref{w}.

\subsection{Series expansion of the far-field loading}
As shown by equation \eqref{exp_anal_function_Faber}, Faber polynomials constitute a basis for analytic functions in $\overline{\Om}$. Specifically, we can expand the functions $h(z)$ and ${l(z)}$ in the complex representation \eqref{eqnd:complex_repr_u0} of the far-field loading $u_0(z)$ as
\begin{equation}\label{h_l}
h(z)=\sum_{m=0}^{\infty}A_mF_m(z)\,,\qquad {l(z)}=\sum_{m=0}^{\infty}B_m{F_m(z)}
\end{equation}
for some complex coefficients $A_m$ and $B_m$ to be determined by using equation \eqref{fabercoeff}. 
Hence, the far-field loading $u_0(z)$ in \eqref{eqnd:complex_repr_u0} can be written as
\begin{equation}\label{u_0_complete_expansion}
2u_0(z)=\kappa \sum_{m=0}^{\infty}A_mF_m(z)-z\sum_{m=1}^{\infty}\overline{A_m}\,\overline{F_m'(z)}-\sum_{m=0}^{\infty}\overline{B_m}\,\overline{F_m(z)}\quad\mbox{in }\overline{\Om}.
\end{equation}

\subsection{Series expansion of the single-layer potential}\label{Section_Series_exp_single_layer_pot}

To obtain a series expansion of the single-layer potential $S[\varphi](z)$, we expand the density function $\varphi$ on $\p\Om$ in terms of the density basis \eqref{exterior_basis} on $\p\Om$:
\begin{equation}\label{varphi_expansion}
\varphi(z)=\sum_{m=1}^{\infty}\left(s_m^{(1)}+\mathrm{i}s_m^{(2)}\right)\varphi_{-m}+\left(s_m^{(3)}+\mathrm{i}s_m^{(4)}\right)\varphi_{m},
\end{equation}
where $s_m^{(j)}$, $m\in\mathbb{N}$, $j=1,2,3,4$, are real coefficients. 
The constant term ($m=0$) is zero due to the equilibrium condition \eqnref{cond:phi} with $j=1,2$, and the condition with $j=3$ implies that
$
\mbox{Im}\left(\int_{\p\Om}\varphi(z)\overline{z}d\sigma(z)\right)=0,
$
i.e.,
\beq\label{s_1_4:cond}
s_1^{(4)}=-\mbox{Im}\left(\sum_{m=1}^\infty \left(s_m^{(1)}+\mathrm{i}s_m^{(2)}\right)\overline{a_m}\right).\eeq
From now on, we assume $\gamma=1$ for the sake of simplicity.

Notice that, by plugging \eqref{varphi_expansion} into \eqref{singlelayer:simple}, the single-layer potential $S[\varphi](z)$ can be expressed in in terms of the basis functions \eqref{exterior_basis} as follows:
\begin{align}
2S[\varphi](z)\notag
&=2{\alpha_1}L[\varphi](z)-\alpha_2 z\, \overline{\mathcal{C}[\varphi](z)}+\alpha_2\overline{\mathcal{C}[\overline{\zeta}\varphi](z)}\quad\mbox{in }\CC\\\notag
&=\alpha_1  \sum_{m=1}^{\infty}\bigg[\left(s_m^{(1)}+\mathrm{i}s_m^{(2)}\right)2{L}[\varphi_{-m}](z)+\left(s_m^{(3)}+\mathrm{i}s_m^{(4)}\right)2{L}[\varphi_m](z)\bigg]\\\notag
&-\alpha_2\, z\sum_{m=1}^\infty \left[\, \overline{\left(s_m^{(1)}+\mathrm{i}s_m^{(2)}\right)} \, \overline{\mathcal{C}[\varphi_{-m}](z)}\, +  \overline{\left(s_m^{(3)}+\mathrm{i}s_m^{(4)}\right)}\, \overline{\mathcal{C}[\varphi_m](z)}\, \right] \\\label{Svarphi:expan}
&+\alpha_2\sum_{m=1}^{\infty}\left[\, \overline{\left(s_m^{(1)}+\mathrm{i}s_m^{(2)}\right)}\, \overline{\mathcal{C}[\overline{\zeta}\varphi_{-m}](z)}\, +\overline{\left(s_m^{(3)}+\mathrm{i}s_m^{(4)}\right)}\, \overline{\mathcal{C}[\overline{\zeta}\varphi_{m}](z)}\, \right].
\end{align}
The operator $L$ given by \eqnref{singlelayer:real}, which corresponds to the single-layer potential of the Laplacian, admits the following series expansion \cite{Jung:2018:SSM}: for each $m\in\NN$,
\begin{align}
L\left[ \varphi_{-m}\right](z)
&=\begin{cases}
\ds -\frac{1}{2m}\, \overline{F_m(z)},\quad&z\in\overline{\Om},\\[2mm]
\ds -\frac{1}{2m}\left(\, \overline{F_m(z)}-\overline{w^m}+{w^{-m}}\right),\quad&z=\Psi(w)\in\CC\setminus\overline{\Om},
\end{cases}\\
L\left[ \varphi_m\right](z)
&=\begin{cases}
\ds -\frac{1}{2m}F_m(z),\quad&z\in\overline{\Om},\\[2mm]
\ds -\frac{1}{2m}\left(F_m(z)-w^m+\overline{w^{-m}}\right),\quad&z=\Psi(w)\in\CC\setminus\overline{\Om}.
\end{cases}
\end{align}
Likewise, the explicit computation of the integral operators $\mathcal{C}[\varphi_{\pm m}](z)$ and ${\mathcal{C}[\overline{\zeta}\varphi_{\pm m}](z)}$, whose general expression is given by \eqref{mathcal_C}, depends on whether $z\in\mathbb{C}\setminus\overline{\Omega}$ or $z\in\overline{\Omega}$, as will be shown in  Lemma \ref{lemma:C:exp:inside} and Lemma \ref{lemma:operato:ext}.

%%%%%%%
\subsubsection{Series expansion of the single-layer potential inside the inclusion}

For notational convenience, we define the following polynomial functions
\begin{align}\label{def:twF}
\widetilde{F}_k(z)&:=
\begin{cases}
\ds\frac{1}{k}{F_k'(z)}\quad&\mbox{for }k\geq 1,\\
\ds 0\quad&\mbox{for }k\leq0.
\end{cases}
\end{align}
\begin{lemma}\label{lemma:C:exp:inside}
	For each $m\in\NN$, we have
	\begin{align*}
	\mathcal{C}\left[\varphi_{\pm m}\right](z)&=-\wtF_{\pm m}(z),\quad
	\mathcal{C}\left[\overline{\zeta}\,\varphi_{\pm m}\right](z)
	= -\sum_{k=-1}^\infty \overline{a_k}\, \wtF_{k\pm m}(z)\quad\mbox{for }z\in\Om.
	\end{align*}
\end{lemma}
\begin{proof}
	We parametrize $\p\Om$ by $\zeta=\Psi(\eta)=\Psi(e^{\mathrm{i}\theta})\in\p\Om.$
	From the definition \eqref{mathcal_C} of the integral operator $\mathcal{C}$ and \eqnref{fund:deri:z}, we have
	\begin{align}\label{mathcal:C:pm}
	\mathcal{C}[\varphi_{\pm m}](z)
	&=\frac{1}{2\pi}\int_0^{2\pi}\left(-\sum_{n=1}^\infty\frac{1}{n}F_n'(z)\eta^{-n}\right)\eta^{\pm m}d\theta\\\notag
	&=-\frac{1}{m}F_m'(z)\mbox{ or }0.
	\end{align}
	Indeed, for a fixed $z\in\Om$ we have
	\begin{align}\notag
	\mathcal{C}[\varphi_{\pm m}](z)
	\ds&=\lim_{t\rightarrow 0^+}\frac{1}{2\pi}\int_{\p \Om}\frac{\varphi_{\pm m}(\zeta)}{z-\Psi((1+t)\eta)}\,\mathrm{d}\sigma(\zeta)\\\label{mathcal:C:pm:validation}
	\ds&= \lim_{t\rightarrow 0^+}\frac{-1}{2\pi}\int_{\p \Om} \sum_{n=1}^\infty\frac{1}{n}F_n'(z)((1+t)\eta)^{-n}\varphi_{\pm m}(\zeta)\,\mathrm{d}\sigma(\zeta)
	\end{align}
	As the power series in $w$ in \eqnref{fund:deri:z} is convergent for $|w|>\gamma$, we can exchange the order of the integral and summation in \eqnref{mathcal:C:pm:validation}, and thus in \eqnref{mathcal:C:pm}.

	By using the Laurent series expansion \eqref{eqn:extmapping} of $\Psi$, we then get
	\begin{align}
	\mathcal{C}\left[\overline{\zeta}\varphi_{\pm m}\right](z)
	&=\mathcal{C}\left[\overline{\Psi(\eta)}\eta^{\pm m}\frac{1}{h}\right](z)\notag \\
	&=\mathcal{C}\left[\sum_{k=-1}^\infty \overline{a_k}\eta^{k\pm m}\frac{1}{h}\right](z)
	=\sum_{k=-1}^\infty \overline{a_k}\,\mathcal{C}\left[\eta^{k\pm m}\frac{1}{h}\right](z).\label{C_overline_zeta_phinegm}
	\end{align}
	Recall that the map $\Psi(\rho,\theta)$ is $C^1$ on $[\rho_0,\infty)\times[0,2\pi)$.
	Since the Fourier series of a continuously differential function is uniformly convergent, we can exchange the order between the operator $\mathcal{C}$ and the summation in \eqnref{C_overline_zeta_phinegm}.
	This completes the proof.
\end{proof}
By plugging the relations in Lemma \ref{lemma:C:exp:inside} into \eqnref{Svarphi:expan}, we obtain the following theorem. 

\begin{theorem}
	The single-layer potential $S[\varphi](z)$ with the density function $\varphi$ given by \eqnref{varphi_expansion} admits the  following expansion, for $z\in\overline{\Om}$:
	{
		\begin{align}
		2 S[\varphi](z)
		=-&\alpha_1\sum_{m=1}^\infty \frac{1}{m} \Big(s_m^{(3)}+\mathrm{i}s_m^{(4)}\Big){F_m(z)}
		+\alpha_2z\sum_{m=1}^{\infty}\overline{\Big(s_m^{(3)}+\mathrm{i}s_m^{(4)}\Big)}\overline{\wtF_m(z)}\nonumber\\
		-&\alpha_1\sum_{m=1}^\infty \frac{1}{m} \Big(s_m^{(1)}+\mathrm{i}s_m^{(2)}\Big)\notag
		\overline{F_m(z)}\\
		-&\alpha_2\sum_{m=1}^\infty   \bigg[ \overline{\Big(s_m^{(1)}+\mathrm{i}s_m^{(2)}\Big)}\sum_{k=m+1}^\infty a_k \overline{\wtF_{k-m}(z)}+ \overline{\Big(s_m^{(3)}+\mathrm{i}s_m^{(4)}\Big)}\sum_{k=-1}^\infty a_k \overline{\wtF_{k+m}(z)}\bigg].\label{S_F_m_F'_m}
		\end{align}
	}
\end{theorem}

The derivative of the $m$-th Faber polynomial is a linear combination of the Faber polynomials of lower order, that is
\begin{equation}\label{F'm_general}
F'_m(z)=\sum_{j=0}^{m-1}\gamma_{m,j}F_j(z),
\end{equation}
where the coefficients $\gamma_{m,j}$, $j=0,\dots,m-1$, depend on the conformal mapping coefficients $a_k$ (see \eqref{eqn:extmapping}). In matrix form, \eqnref{F'm_general} takes the following expression
\beq\label{Fderi:F}
\mathbf{F}'=\boldsymbol{\Gamma}\,\mathbf{F}+\boldsymbol{\gamma_0},
\eeq
where
\beq\label{Gamma_F}
\mathbf{F}:=\left[\begin{array}{c}F_1\\
	F_2\\
	F_3\\
	\vdots
\end{array}
\right],\,\quad 
\boldsymbol{\Gamma}:=\left[\begin{array}{cccc}
	0&0&0&\cdots\\
	\gamma_{21}&0&0&\cdots\\
	\gamma_{31}&\gamma_{32}&0&\cdots\\
	\vdots&\vdots&\vdots&\ddots
\end{array}\right] \quad \text{and}\quad \boldsymbol{\gamma_0}:=\left[\begin{array}{c}\gamma_{10}\\
	\gamma_{20}\\
	\gamma_{30}\\
	%\vdots\\
	%F_{m-1}\\
	\vdots
\end{array}
\right].
\eeq
By using \eqnref{F'm_general} or, equivalently \eqref{Fderi:F}, we can express $S[\varphi](z)$ \eqref{S_F_m_F'_m} in terms of Faber polynomials only (not their derivatives).

\subsubsection{Series expansion of the single-layer potential outside the inclusion}
We define 
\begin{align}
\label{def:Gk}
G_k(w)&:= \ds \frac{w^{k-1}}{\Psi'(w)}\quad \mbox{for }k\in\ZZ,\ |w|>\gamma,
\end{align}
which approximates $\wtF_k(z)$ given in \eqnref{def:twF}, acoording to \eqnref{eqn:Faberdefinition}:
\beq\notag
\widetilde{F}_k(\Psi(w))=G_k(w)+O\left(\frac{1}{|w|}\right)\quad\mbox{for each }k\in\ZZ.
\eeq

\begin{lemma}\label{lemma:operato:ext}
	
	For each $m\in\NN$, it holds that for $z=\Psi(w)\in\CC\setminus\overline{\Om}$,
	\begin{align}\label{C1:exterior:simple}
	\mathcal{C}\left[\varphi_{\pm m}\right](z)
	&=-\Big(\wtF_{\pm m}(z)-G_{\pm m}(w)\Big),\\ \label{C2:exterior:simple}
	\mathcal{C}\left[\overline{\zeta}\,\varphi_{\pm m}\right](z)
	&=- \sum_{k=-1}^\infty {\overline{a_k}}\Big( \wtF_{k\pm m}(z)-G_{k\pm m}(w)\Big).
	\end{align}
	%with $\wtF_k(z)$ and $G_k(w)$ given by \eqnref{def:twF} and \eqnref{def:Gk}.
\end{lemma}
\begin{proof}
	Like in the proof of Lemma \ref{lemma:C:exp:inside}, we parameterize $\p\Om$ by $\zeta=\Psi(\eta)=\Psi(e^{\mathrm{i}\theta})\in\p\Om$.
	From \eqnref{eqn:Fabergenerating}, \eqnref{eqn:Faberdefinition}, and the properties of Grunsky coefficients, we have
	\begin{align}
	\mathcal{C}\left[\varphi_{\pm m}\right](z)
	&=\lim_{t\rightarrow 0^+}\frac{1}{2\pi}\int_0^{2\pi}\frac{1}{\Psi(w)-\Psi(\wteta)}\, \eta^{\pm m}\,d\theta,\quad\wteta=(1+t)\eta, \notag \\ \label{C:phi:int1}
	&=\lim_{t\rightarrow 0^+}\frac{1}{2\pi}\int_0^{2\pi}\sum_{n=1}^\infty F_n(\Psi(\wteta))\frac{w^{-n-1}}{\Psi'(w)}\, \eta^{\pm m}\,d\theta\\\label{C:phi:int2}
	&= \lim_{t\rightarrow 0^+}\frac{1}{2\pi}\sum_{n=1}^\infty\int_0^{2\pi}\bigg(\wteta^n+\sum_{k=1}^\infty\frac{n}{k}c_{k,n}\wteta^{-k}\bigg)\frac{w^{-n-1}}{\Psi'(w)}\, \eta^{\pm m}\,d\theta,
	\end{align}
	We assume $1+t<|w|$ so that from \eqnref{Faber:bound}, the infinite series in \eqnref{eqn:Fabergenerating} is uniform with respect to $z=\Psi(\wteta)$ with $|\wteta|=1+t$. Hence, one can exchange the order of integration and summation in \eqnref{C:phi:int1}. Furthermore, from \eqnref{c:bound}, the integrand in \eqnref{C:phi:int2} is uniformly convergent in $\eta$ so that one can again exchange the order of integration and summation in \eqnref{C:phi:int2}. 
	Therefore, we have
	\begin{align*}
	\mathcal{C}\left[\varphi_{m}\right](z)
	&=\frac{1}{\Psi'(w)}\sum_{n=1}^\infty \frac{n}{m} c_{m,n}w^{-n-1}\\
	&=-\frac{1}{m}\frac{1}{\Psi'(w)}\frac{d}{dw}\Big(F_m(\Psi(w))-w^m\Big)
	\end{align*}
	and
	\begin{align*}
	\mathcal{C}\left[\varphi_{-m}\right](z)
	&=\frac{w^{-m-1}}{\Psi'(w)}.
	\end{align*}
	%By recalling from \eqref{C_overline_zeta_phinegm} that
	%\begin{align*}
	%\mathcal{C}\left[\overline{\zeta}\varphi_{\pm m}\right](w)
	%&=\sum_{k=-1}^\infty \overline{a_k}\,\mathcal{C}\left[\eta^{k\pm m}\frac{1}{h}\right](w),	
	%	\end{align*}
	From the definition of $\wtF_k(z)$ and $G_k(w)$, we have \eqnref{C1:exterior:simple}.
	Using \eqref{C_overline_zeta_phinegm}, one can then easily find the the expansions of $\mathcal{C}\left[\overline{\zeta}\,\varphi_{\pm m}\right](w)$.
\end{proof}

Following the same steps taken for the series expansion of the single layer potential in the interior of the inclusion, one has just to plug \eqnref{C1:exterior:simple} and \eqnref{C2:exterior:simple} into \eqnref{Svarphi:expan} to find the expansion for $S[\varphi](z)$ in \eqref{single_layer} with $z=\Psi(w)$ in the exterior of the inclusion. The result is presented in the following theorem.
\begin{theorem}
	The single-layer potential $S[\varphi](z)$ with the density function $\varphi$ given by \eqnref{varphi_expansion} admits the  following expansion, for $z=\Psi(w)\in\CC\setminus\overline{\Om}$:

	\begin{equation}\label{S_exterior} 
	S_{ext}[\varphi](z)=-\alpha_1 v_1(w)+\alpha_2 \Psi(w)\, \overline{ v_2(w)}-\alpha_2 \overline{v_3(w)},	\end{equation}
	where 
	\begin{align*}
	v_1&=
	\sum_{m=1}^\infty \frac{1}{m}\bigg[ \Big(s_m^{(1)}+\mathrm{i}s_m^{(2)}\Big)
	\Big(\overline{F_m(z)}-\overline{w^m}+w^{-m}\Big)+ \Big(s_m^{(3)}+\mathrm{i}s_m^{(4)}\Big)
	\Big({F_m(z)}-{w^m}+\overline{w^{-m}}\Big)\bigg],\\[1mm]
	v_2&= \sum_{m=1}^\infty \bigg[{\Big(s_m^{(1)}+\mathrm{i}s_m^{(2)}\Big)}\left(\wtF_{-m}(z)- G_{-m}(w)\right)
	+ {\Big(s_m^{(3)}+\mathrm{i}s_m^{(4)}\Big)}\Big(\wtF_m(z)-G_m(w)\Big)\bigg],\\[1mm]
	%\end{align*}
	%and
	%\begin{align*}
	v_3&=\sum_{m=1}^\infty  \bigg[\Big(s_m^{(1)}+\mathrm{i}s_m^{(2)}\Big)
	\sum_{k=-1}^\infty \overline{a_k}\Big({\wtF_{k-m}(z)} -{G_{k-m}(w)}\Big)
	+ {\Big(s_m^{(3)}+\mathrm{i}s_m^{(4)}\Big)}\sum_{k=-1}^\infty \overline{a_k} \Big({\wtF_{k+m}(z)} -{G_{k+m}(w)} \Big)\bigg]
	\end{align*}
	with $\wtF_k(z)$ and $G_k(w)$ given by \eqnref{def:twF} and \eqnref{def:Gk}, respectively.
\end{theorem}

\subsection{Determination of the density function via the transmission condition}\label{section:determination_s}
{The single-layer potential, \eqnref{single_layer}, is continuous on $\p \Om$. In particular, the expansions \eqnref{S_F_m_F'_m} and \eqnref{S_exterior} coincide when $|w|=1$.}
In order to solve the rigid inclusion problem, we have first to find the density function $\varphi(z)$ that satisfies the transmission condition \eqnref{trans_cond_S_u0}, that is, we have to determine the unknown real coefficients $s_m^{(j)}$, $m\in\mathbb{Z}$, $j=1,2,3,4$ in the expansion \eqref{varphi_expansion} of $\varphi(z)$, by comparing the expansion of the single-layer potential, \eqnref{S_F_m_F'_m}, with the expansion of the far-field loading \eqref{u_0_complete_expansion} via \eqnref{trans_cond_S_u0}. We recall that the real coefficients $c_j$, $j=1,2,3$ in \eqnref{trans_cond_S_u0} are implicitly determined by the equilibrium condition \eqref{equilibrium_boundary} for which the solution of the transmission problem has first to be found. 

The strategy to determine the coefficients $s_m^{(j)}$, $m\in\mathbb{Z}$, $j=1,2,3,4$ is the following. We will first determine the coefficients $s_m^{(3)}$ and $s_m^{(4)}$, $m\in\mathbb{Z}$. Below we will show that,  for $m>1$, such coefficients are completely determined by the coefficients of the series expansion \eqref{u_0_complete_expansion} of the far-field loading $u_0(z)$. The determination of $s_1^{(3)}$ and $s_1^{(4)}$, instead, as well as the determination of the coefficients $s_m^{(1)}$ and $s_m^{(2)}$, $m\in\mathbb{Z}$, requires the knowledge of the constant $c_3$, which we will derive at the end of this section by using the linear dependence of $s_m^{(1)}$ and $s_m^{(2)}$, $m\in\mathbb{Z}$ on $c_3$. 
%Therefore, the density function will be determined up to the knowledge of the two constants $c_1$ and $c_2$, which will finally be calculated by using the equilibrium condition \eqref{equilibrium_boundary} involving the solution $u(z)$ of the problem that will be provided in Section \ref{subsection_Solution}.

\subsubsection{Determination of the coefficients $s_m^{(3)}$ and $s_m^{(4)}$, $m\in\mathbb{Z}$}

By using the transmission condition \eqnref{trans_cond_S_u0} and comparing the analytic part, that is $\phi$ in the decomposition of the form \eqnref{rep_u_kappa}, in \eqnref{S_F_m_F'_m} and \eqref{u_0_complete_expansion}, we get
\begin{align}
s_1^{(3)}+\mathrm{i}s_1^{(4)}&=\frac{1}{\alpha_2}A_1+\frac{1}{\alpha_1}2\mathrm{i}c_3,\label{s_3_4:m1}\\
\label{s_3_4}
s_m^{(3)}+\mathrm{i}s_m^{(4)}&=\frac{1}{\alpha_2} mA_m\quad\mbox{for each }m\geq2.
\end{align}
Hence, all the coefficients $s_m^{(3)}$ and $s_m^{(4)}$ for $m\geq 2$ are explicitly given by the coefficients $A_m$ of the series expansion \eqref{u_0_complete_expansion} of the far-field loading $u_0(z)$. The case $m=1$ is more complex, due to the fact that $c_3$ in \eqref{s_3_4:m1} is unknown. However, it will be determined by solving a linear equation as explained in Section \ref{subsection_c3}.

\subsubsection{Determination of the coefficients $s_m^{(1)}$ and $s_m^{(2)}$, $m\in\mathbb{Z}$}

By using the relations \eqnref{s_3_4:m1} and \eqnref{s_3_4}, as well as \eqnref{u_0_complete_expansion} and \eqnref{S_F_m_F'_m}, the transmission condition \eqnref{trans_cond_S_u0} on $\partial\Omega$ turns into the following.
%	{\color{red}
%	\begin{align}
%	2 S[\varphi](z)
%	=-&\alpha_1\sum_{m=1}^\infty \frac{1}{m} \Big(s_m^{(3)}+\mathrm{i}s_m^{(4)}\Big){F_m(z)}
%	+\alpha_2z\sum_{m=1}^{\infty}\overline{\Big(s_m^{(3)}+\mathrm{i}s_m^{(4)}\Big)}\overline{\wtF_m(z)}\nonumber\\
%	-&\alpha_1\sum_{m=1}^\infty \frac{1}{m} \Big(s_m^{(1)}+\mathrm{i}s_m^{(2)}\Big)\notag
%\overline{F_m(z)}\\
%	-&\alpha_2\sum_{m=1}^\infty   \bigg[ \overline{\Big(s_m^{(1)}+\mathrm{i}s_m^{(2)}\Big)}\sum_{k=m+1}^\infty a_k \overline{\wtF_{k-m}(z)}+ \overline{\Big(s_m^{(3)}+\mathrm{i}s_m^{(4)}\Big)}\sum_{k=-1}^\infty a_k \overline{\wtF_{k+m}(z)}\bigg].\label{S_F_m_F'_m}
%	\end{align}
%	}

{
	\begin{problem*}[A] Find $s_m^{(1)}$ and $s_m^{(2)}$, $m\in\ZZ$ satisfying
		\beq\label{eqn:withJ}
		\mathcal{P}(z)=  c_3 J_1(z)+J_2(z) +C,
		\eeq
		where 
		\beq
		\mathcal{P}(z)
		=-\alpha_1\sum_{m=1}^{\infty}\frac{1}{m}{\left(s_m^{(1)}+\mathrm{i}s_m^{(2)}\right)}\overline{F_m(z)}-\alpha_2\sum_{m=1}^{\infty}\overline{\left(s_m^{(1)}+\mathrm{i}s_m^{(2)}\right)}\sum_{k=m+1}^\infty a_k\overline{\wtF_{k-m}(z)} \label{oper:P}
		\eeq
		and
		\begin{align}
		J_1(z)&={\,-\mathrm{i}\,\frac{2}{\kappa}\sum_{k=0}^\infty a_k \overline{\wtF_{k+1}(z)},\notag }\\
		J_2(z)&=\sum_{m=1}^{\infty}\overline{B_m}\,\overline{F_m(z)}
		+\sum_{m=1}^{\infty}m\overline{A_m}\sum_{k=-1}^\infty  a_k\overline{\wtF_{k+m}(z)},\label{def:J1J2}\\
		C&=-\kappa A_0+\overline{B_0}+2c_1+2\mathrm{i}c_2- 2\mathrm{i} c_3 a_0.\label{def:C}
		\end{align}
	\end{problem*}
}
Note that $\mathcal{P}(z)$ is the only complex function that contains the unknowns  $s_m^{(1)}$ and $s_m^{(2)}$, $m\in\mathbb{Z}$, whereas $J_1(z)$ and $J_2(z)$ are known complex functions. Indeed,  since $\Om$ is given, we can compute the Faber polynomials $F_m(z)$ associated with $\Om$ via \eqref{eqn:Fabergenerating}, as well as their derivatives (or, equivalently, the coefficients $\gamma_{mj}$ in \eqref{F'm_general}). Hence, $J_1(z)$ is determined. 
Since $u_0(z)$ is given, $J_2(z)$ is also completely determined.
%In section \ref{sec:examples}, we solve equations of the form $\mathcal{P}\left[\mathbf{s}\right]=J$ for various examples. 

For the sake of clarity, we adopt a matrix notation. Then, let us denote with $\mathbf{s}$ the vector containing the unknown coefficients $s_m^{(1)}$ and $s_m^{(2)}$, $m\geq 1$, i.e., 
$$\mathbf{s}:=
\left[\begin{array}{c}
s_1^{(1)}+\mathrm{i}s_1^{(2)}\\[1mm]
s_2^{(1)}+\mathrm{i}s_2^{(2)}\\[1mm]
\vdots
\end{array}
\right].$$
To stress the fact that $\mathcal{P}(z)$ is an operator acting on the unknown vector $\mathbf{s}$, let us use the notation $\mathcal{P}\left[\mathbf{s}\right](z)$. To write $\mathcal{P}\left[\mathbf{s}\right](z)$ in matrix form, let us start by the term
{
	$$b_m := \sum_{k=m+1}^\infty a_k\overline{\wtF_{k-m}(z)}=\sum_{k=m+1}^\infty\frac{a_k}{k-m}\overline{F'_{k-m}(z)},\quad m\geq1,$$ }
which, in matrix notation, reads
\begin{equation}\label{b_matrix_form_F'}
\mathbf{b}=\mathbf{A\,D}\,\overline{\mathbf{F}'},
\end{equation}
where $\mathbf{F}$ is given by \eqref{Gamma_F}, and
\begin{equation}\label{b_A_D_def}\mathbf{b}:=\left[\begin{array}{c}
b_1\\
b_2\\
b_3\\
\vdots
\end{array}
\right], \quad\mathbf{A}:=\left[\begin{array}{cccc}
a_2&a_3&a_4&\cdots\\
a_3&a_4&a_5&\cdots\\
a_4&a_5&a_6&\cdots\\
\vdots&\vdots&\vdots&\ddots
\end{array}
\right], \quad\mathbf{D}:=
\left[\begin{array}{cccc}
1&0&0&\cdots\\
0&\frac{1}{2}&0&\cdots\\
0&0&\frac{1}{3}&\cdots\\
\vdots&\vdots&\vdots&\ddots
\end{array}\right].
\end{equation}
By using \eqnref{Fderi:F}, \eqref{b_matrix_form_F'} turns into
\begin{equation}\label{b:decomp}
\mathbf{b}
=\mathbf{A\,D}\left(\overline{\boldsymbol{\Gamma}}\,\overline{\mathbf{F}}+\overline{\boldsymbol{\gamma_0}}\right),
\end{equation}
and, consequently, the operator $\mathcal{P}\left[\mathbf{s}\right](z)$ in \eqnref{oper:P} reads
\beq
\mathcal{P}\left[\mathbf{s}\right](z)
=-\alpha_2\left(\kappa\,\mathbf{s}^T\mathbf{D}+\overline{\mathbf{s}}^T\mathbf{A\,D\,}\overline{\boldsymbol{\Gamma}}\right)\overline{\mathbf{F}}-\alpha_2\,\overline{\mathbf{s}}^T\mathbf{A\,D\,}\overline{\boldsymbol{\gamma_0}}.
\eeq
In order to write the right hand side of \eqref{eqn:withJ} in matrix form, let us introduce the vector $\mathbf{y}$, defined as the vector containing the coefficients of the expansion of the function 
$$J(z)=c_3J_1(z)+J_2(z)+C$$
with respect to the basis functions $\overline{F_m}$ (up to the coefficient $-\alpha_2$), that is
\beq\label{def:mathbfy}
J(z)=-{\alpha_2}\,\mathbf{y}^T\overline{ \mathbf{F}}-\alpha_2\,j_0.
\eeq
Note that the unknown constants $c_1$ and $c_2$ are incorporated in the constant $j_0$, whereas the constant $c_3$ appears linearly both in $\mathbf{y}$ and $j_0$, due to the linear dependence of $J(z)$ on $c_3$. 
Hence, equation \eqref{eqn:withJ}  turns into
\beq\label{full_eq_s}
\left(\kappa\,\mathbf{s}^T\mathbf{D}+\overline{\mathbf{s}}^T\mathbf{A\,D\,}\overline{\boldsymbol{\Gamma}}\right)\overline{\mathbf{F}}+\overline{\mathbf{s}}^T\mathbf{A\,D\,}\overline{\boldsymbol{\gamma_0}}=\mathbf{y}^T\overline{ \mathbf{F}}+j_0
\eeq
and by comparing the coefficients of the series expansion with respect to the Faber polynomials, we get
\beq\kappa\,{\mathbf{s}}^T\mathbf{D}+\overline{\mathbf{s}}^T\mathbf{A\,D\,}\overline{\boldsymbol{\Gamma}}
=\mathbf{y}^T\label{matrix_form_eqn_s},\eeq
\beq\label{eqn:def:j0}
\overline{\mathbf{s}}^T\mathbf{A\,D\,}\overline{\boldsymbol{\gamma_0}}=j_0.
\eeq
%Notice that the vector $\mathbf{y}$ is determined up to the constant $c_3$ that we will determine in terms of $c_1$ and $c_2$ in Section \ref{subsection_c3}. 

As we will show explicitly in Section \ref{sec:examples}, for domains of order up to 2, the matrix $\mathbf{A}$ defined in \eqref{b_A_D_def} is the zero matrix and, therefore, the coefficients $s_m^{(1)}$ and $s_m^{(2)}$ are easily found in terms of the unknown constant $c_3$  as \eqref{matrix_form_eqn_s} reduces to $\kappa\,\mathbf{D\,s}={\mathbf{y}}$, in which the matrix $\mathbf{D}$ defined in \eqref{b_A_D_def}  is invertible, and the vector ${\mathbf{y}}$ incorporates $c_3$.
% \textcolor{red}{When the domain is of order 3, then the only non-zero entry of the matrix $\mathbf{A}$ is the 11-element, thus leading to equation \eqref{sol_order3}. FIX THE EXAMPLES SECTION AND THEN RECHECK THIS STATEMENT}.
For domains of higher order, the coefficients $s_m^{(1)}$ and $s_m^{(2)}$ are found as follows.
By taking the transpose of \eqref{matrix_form_eqn_s} and then the complex conjugate of the transpose, we respectively get
\begin{align}
\kappa\,\mathbf{D}\,\mathbf{s}+\overline{\boldsymbol{\Gamma}}^T\mathbf{D}\,{\mathbf{A}}\,\overline{\mathbf{s}}&={\mathbf{y}},\notag\\
{\boldsymbol{\Gamma}^T}{\mathbf{D}}\,\overline{\mathbf{A}}\,\mathbf{s} + \kappa\,\mathbf{D}\,\overline{\mathbf{s}}&= \overline{\mathbf{y}},\label{eqn:x:conj}
\end{align}
{which leads to the reformulation of Problem (A) in a block matrix form:
	\begin{problem*}[A$^\prime$] Find $\mathbf{s}$ satisfying
		\beq
		\left[\begin{array}{cc}
			\kappa\,\mathbf{D}&\overline{\boldsymbol{\Gamma}}^T\mathbf{D\,A}\\[3mm]
			\boldsymbol{\Gamma}^T\mathbf{D}\,\overline{\mathbf{A}}&\kappa\,\mathbf{D}
		\end{array}
		\right]
		\left[
		\begin{array}{c}
			\mathbf{s}\\[3mm]
			\overline{\mathbf{s}}
		\end{array}
		\right]
		=\left[
		\begin{array}{c}
			\mathbf{y}\\[3mm]
			\overline{\mathbf{y}}
		\end{array}
		\right],\label{matrix_block_system}
		\eeq
		where $\mathbf{y}$ is given by \eqnref{def:mathbfy}.
		%We then determine $c_1$ and $c_2$ to satisfy \eqnref{eqn:def:j0}.
	\end{problem*}
	The invertibility of the block matrix in \eqref{matrix_block_system} has to be assessed case by case. Specifically, it is invertible if $\kappa^2 \mathbf{I}-(\boldsymbol{\Gamma}^T\mathbf{D}\,\overline{\mathbf{A}}\,\mathbf{D}^{-1})(\overline{\boldsymbol{\Gamma}}^T\,\mathbf{D}\,\mathbf{A}\mathbf{D}^{-1})$ is invertible (note that $\mathbf{D}$, defined in \eqref{b_A_D_def}, is invertible).} In Section \ref{sec:examples} we will show some explicit examples in which it can be proven the matrix is invertible. 
	{Again, we stress the fact that the vector $\mathbf{y}$ depends linearly on the constant $c_3$, i.e., $\mathbf{y}=c_3\mathbf{y}_1+\mathbf{y}_2$ with $\mathbf{y}_1$ and $\mathbf{y}_2$ corresponding to $J_1(z)$ and $J_2(z)$, respectively. Hence, the solution \textbf{s}, upon inversion of the block matrix, is
\begin{equation}\label{s_linear_dep_c3}
\mathbf{s}=c_3\mathbf{u}_1+\mathbf{u}_2,
\end{equation}
where $\mathbf{u}_1$ and $\mathbf{u}_2$ are solutions to \eqnref{matrix_block_system} with $\mathbf{y}_1$ and $\mathbf{y}_2$ in the place of $\mathbf{y}$, respectively. 
}

\subsubsection{Determination of the constants $c_1, c_2, c_3$}\label{subsection_c3}

The linear dependence of $\mathbf{s}$ on the constant $c_3$ in \eqref{s_linear_dep_c3} plays a crucial role in the determination of $c_3$. Indeed, by combining \eqnref{s_1_4:cond} and \eqnref{s_3_4:m1} we get
\beq\label{S14:cond}
\mbox{Im}\left(\mathbf{s}\cdot \overline{\mathbf{a}}\right)=-\frac{\mbox{Im}(A_1)}{\alpha_2}-\frac{2c_3}{\alpha_1},\eeq
where \textbf{a} is the vector containing the coefficients of the conformal mapping \eqref{eqn:extmapping}, that is, $\mathbf{a}^{T}=[a_1; a_2; a_3;\dots]$. By using \eqref{s_linear_dep_c3}, we get
\begin{equation}\label{c_3}
c_3=\kappa\frac{-{\mbox{Im}(A_1)}-\alpha_2\mbox{Im}(\mathbf{u}_2\cdot \overline{ \mathbf{a}})}{{2}+{\alpha_1}\mbox{Im}(\mathbf{u}_1\cdot\overline{\mathbf{a}})}.
\end{equation}
Once $c_3$ and, consequently the vector $\mathbf{s}$, are known, the coefficients $c_1$ and $c_2$ are then found by using equation \eqref{eqn:def:j0}. 

\subsection{Solution expansion in the exterior of the inclusion}\label{subsection_Solution}
Due to the ansatz \eqref{relation_u_u0_S}, the displacement field outside the inclusion is given by 
\begin{equation}
\label{solution}
u(z)=u_0(z)+\Scal_{ext}[\varphi](z)\quad \text{with }z=\Psi(w)\in\mathbb{C}\setminus\overline{\Omega},
\end{equation} 
where $u_0(z)$ is the given far-field loading, and $\Scal_{ext}[\varphi](z)$ is the expansion \eqref{S_exterior} of the single layer potential in the exterior of the inclusion, in which the coefficients $s_m^{(j)}$, $j=1,2,3,4$, $m\in\mathbb{Z}$ are determined as explained in Section \ref{section:determination_s}. 

%The last step in the determination of the explicit expansion of $u(z)$ consists in the determination of the constant $c_1$ and $c_2$ by means of the equilibrium condition \eqref{equilibrium_boundary}. 
%\begin{align}
%u(z)=&u_0(z)+\Scal[\varphi](z)\notag\\
%=&u_0(z)+\frac{1}{2}\left(v_1(z)+v_2(w)\right)
%\end{align}

%{In the next section, we will give the explicit analytical expression of the coefficients $s_m^{(j)}$, $m\in\NN$, $j=1,2,3,4$, for the case of inclusions for which the corresponding conformal mapping \eqref{eqn:extmapping} has at most degree 3. Furthermore, we will show the procedure to determine the coefficients for inclusions with higher degrees for which, instead, some numerical computations are required. }

\section{Examples}\label{sec:examples}

In order to provide an explicit expression for the solution \eqref{solution} of the transmission problem, the coefficients $s_m^{(j)}$, $j=1,2,3,4$ have to be determined explicitly: equation \eqref{s_3_4} allows one to obtain all $s_m^{(3)}$ and $s_m^{(4)}$ for $m>1$ explicitly in terms of the coefficients $A_m$ of the series expansion \eqref{u_0_complete_expansion} of the far-field loading $u_0(z)$, whereas the case $m=1$, see equation \eqref{s_3_4:m1}, requires the knowledge of the constant $c_3$, which can be found when the block matrix in \eqref{matrix_block_system}  is invertible. In such a case, the coefficients $s_m^{(1)}$ and $s_m^{(2)}$, $m\in\mathbb{Z}$ are completely determined as well. The inversion of the matrix in \eqref{matrix_block_system}  is ensured for domains of order up to 3, whereas for domains of higher degree some numerical computations are necessary.

\subsection{Elliptic inclusion with an arbitrary far-field loading }

Let us start by considering the case in which the inclusion is an ellipse. Consequently, the exterior conformal mapping \eqref{eqn:extmapping} takes the following expression:
\begin{equation*}
\Psi(w)=w+\frac{a}{w}.
\end{equation*}
%The corresponding simplified version of the single layer potential \eqref{S_F_m_F'_m} is given by 
%\begin{align}
%2 S[\varphi](z)=& (\kappa\alpha_2-\alpha_1)\sum_{m=1}^{\infty}\left[c_{-m}\left(s_m^{(1)}+\mathrm{i}s_m^{(2)}\right)+c_m\left(s_m^{(3)}+\mathrm{i}s_m^{(4)}\right)\right]\nonumber\\
%-&\kappa\alpha_2\sum_{m=1}^{\infty}\frac{1}{m}\left(s_m^{(3)}+\mathrm{i}s_m^{(4)}\right)F_m(z)+\alpha_2z\sum_{m=1}^{\infty}\frac{1}{m}\overline{\left(s_m^{(3)}+\mathrm{i}s_m^{(4)}\right)}\overline{F'_m(z)}\nonumber\\
%-&\alpha_1\sum_{m=1}^{\infty}\frac{1}{m}{\left(s_m^{(1)}+\mathrm{i}s_m^{(2)}\right)}\overline{F_m(z)}\nonumber-\alpha_2\frac{a}{2}\overline{\left(s_1^{(3)}+\mathrm{i}s_1^{(4)}\right)} \overline{F_{2}'(z)}\\
%-&\alpha_2\sum_{m=2}^{\infty}\overline{\left(s_m^{(3)}+\mathrm{i}s_m^{(4)}\right)}\left(  \frac{{1}}{m-1}\overline{F_{m-1}'(z)}+\frac{{a}}{m+1}\overline{F_{m+1}'(z)}\right).\label{S_ellipse}
%\end{align}

For this case, one can obtain a simple formula for the Faber polynomials and their derivatives by applying the recursive formula \eqref{eqn:Faberrecursion}, with $a_{-1}=1$, $a_1=a$, and $a_j=0$ for all $j\neq \pm 1$. Indeed, after some algebra, one gets
\begin{align*}
\nonumber&F_0(z)=1\\&
F_m(z)=\frac{1}{2^m}\left[\left(z+\sqrt{z^2-4a}\right)^m+\left(z-\sqrt{z^2-4a}\right)^m\right],\quad m=1,2,\dots
\end{align*}
Upon derivation one obtains
\begin{equation*}
F_m'(z)=\frac{m}{2^m\sqrt{z^2-4a}}\left[\left(z+\sqrt{z^2-4a}\right)^m-\left(z-\sqrt{z^2-4a}\right)^m\right]
\end{equation*}
and by using the formula
$$C^m-D^m=(C-D)\sum_{j=0}^{m-1}C^{m-j-1}D^j$$
one can derive an explicit formula for $F_m'(z)$. After some lengthy algebra, we obtain
\begin{equation*}
F_m'(z)=
\begin{cases}
m\sum_{j=0}^{\frac{m-1}{2}}a^{\frac{m-1}{2}-j}F_{2j}\quad&m\mbox{ odd},\\
\\
m\sum^{\frac{m}{2}}_{j=1}a^{\frac{m}{2}-j}F_{2j-1}\quad&m\mbox{ even}.
\end{cases}
\end{equation*}
Therefore, the matrix $\boldsymbol{\Gamma}$ and the vector $\boldsymbol{\gamma_0}$ in the matrix equation \eqref{Fderi:F} take the following explicit expression:
{
	\beq\label{Gamma_F_ellipse}
	\boldsymbol{\Gamma}=\left[\begin{array}{ccccccccc}
		\hskip .16cm 0 \hskip .16cm&\hskip .16cm 0 \hskip .16cm&\hskip .16cm 0 \hskip .16cm&\hskip .16cm 0 \hskip .16cm&\hskip .16cm 0 \hskip .16cm &\hskip .16cm 0 \hskip .16cm &\hskip .16cm 0 \hskip .16cm&\hskip .16cm 0 \hskip .16cm &\hskip .1cm\cdots \hskip .1cm\\[.02cm]
		2&0&0&0&0&0&0&0&\cdots\\[.02cm]
		0&3&0&0&0&0&0&0&\cdots\\[.02cm]
		4a&0&4&0&0&0&0&0&\cdots\\[.02cm]
		0&5a&0&5&0&0&0&0&\cdots\\[.02cm]
		6a^2&0&6a&0&6&0&0&0&\cdots\\[.02cm]
		0&7a^2&0&7a&0&7&0&0&\cdots\\[.02cm]
		8a^3&0&8a^2&0&8a&0&8&0&\cdots\\[.02cm]
		0&9a^3&0&9a^2&0&9a&0&9&\cdots\\[.02cm]
		%\vdots&\vdots&\vdots&\ddots&\vdots&\ddots\\
		%\gamma_{m,1}&\gamma_{m,2}&\gamma_{m,3}&\cdots&\gamma_{m,m-1}&\cdots\\
		\vdots&\vdots&\vdots&\vdots&\vdots&\vdots&\vdots&\vdots&\ddots
	\end{array}\right] \quad \text{and}\quad 
	\boldsymbol{\gamma_0}=\left[\begin{array}{c}
		1\\[.02cm]
		0\\[.02cm]
		3a\\[.02cm]
		0\\[.02cm]
		5a^2\\[.02cm]
		0\\[.02cm]
		7a^3\\[.02cm]
		0\\[.02cm]
		9a^4\\[.02cm]
		\vdots
	\end{array}
	\right].
	\eeq
}

%Let us consider the transmission condition \eqref{trans_cond_S_u0}, where the far-field loading $u_0(z)$ is given by \eqref{u_0_complete_expansion}
%and let us compare the terms in \eqref{u_0_complete_expansion} and \eqref{S_ellipse}. 

As already mentioned, the coefficients $s_m^{(3)}$ and $s_m^{(4)}$, $m>1$, are explicitly given by \eqref{s_3_4} in terms of the coefficients $A_m$ of the series expansion \eqref{u_0_complete_expansion} of the far-field loading $\mathbf{u}_0$, here supposed to be arbitrary. 

{For what concerns the coefficients $s_m^{(1)}$ and $s_m^{(1)}$, we have that, for this case, 
	the operator $\mathcal{P}\left[\mathbf{s} \right](z)$ reads
	$$\mathcal{P}\left[\mathbf{s} \right](z)=-\alpha_1\sum_{m=1}^{\infty}\frac{1}{m}{\left(s_m^{(1)}+\mathrm{i}s_m^{(2)}\right)}\overline{F_m(z)},$$
	whereas the known complex functions in the right-hand side of equations \eqref{eqn:withJ} take the following expression:
	\begin{align}
	J_1(z)&={-\mathrm{i}\frac{a}{\kappa}\, \overline{F_2'(z)},}\\
	J_2(z)&=\sum_{m=1}^{\infty}\overline{B_m}\,\overline{F_m(z)}
	+\frac{a}{2}\overline{A_1}\, \overline{F_{2}'(z)}
	+\sum_{m=2}^{\infty}m\overline{A_m}\left(  \frac{{\overline{F_{-1+m}'(z)}}}{-1+m}+\frac{{a\,\overline{F_{1+m}'(z)}}}{1+m}\right).\label{J_ellipse}
	\end{align}
	Since $\mathbf{A}=\mathbf{0}$, the problem in matrix form, Problem (A$^\prime$) \eqnref{matrix_block_system}, turns into
	\beq\notag
	\left[\begin{array}{cc}
		\kappa\,\mathbf{D}&\mathbf{0}\\[3mm]
		\mathbf{0}&\kappa\,\mathbf{D}
	\end{array}
	\right]
	\left[
	\begin{array}{c}
		\mathbf{s}\\[3mm]
		\overline{\mathbf{s}}
	\end{array}
	\right]
	=\left[
	\begin{array}{c}
		\mathbf{y}\\[3mm]
		\overline{\mathbf{y}}
	\end{array}
	\right],
	\eeq
	in which the matrix \textbf{D}, given by \eqref{b_A_D_def}, is invertible and $\mathbf{y}$ is given by \eqnref{def:mathbfy}.}
Upon inversion, the coefficients $s_m^{(1)}$ and $s_m^{(1)}$ can be written as \eqref{s_linear_dep_c3}, and the constant $c_3$ is found by means of equation \eqref{c_3}. 
%\beq\label{sol_order1}
%-\alpha_1\sum_{m=1}^{\infty}\frac{1}{m}{\left(s_m^{(1)}+\mathrm{i}s_m^{(2)}\right)}\overline{F_m(z)}=J(z),
%\eeq
%where the expression of the known function $J(z)$, given by \eqref{J}, reduces to
%\begin{equation*}
%J(z)=\sum_{m=1}^{\infty}\overline{B_m}\,\overline{F_m(z)}
%+\frac{a}{2}\overline{A_1}\, \overline{F_{2}'(z)}
%+\sum_{m=2}^{\infty}m\overline{A_m}\left(  \frac{{\overline{F_{m-1}'(z)}}}{m-1}+\frac{{a\,\overline{F_{m+1}'(z)}}}{m+1}\right).\label{J_ellipse}
%\end{equation*}

%Then, the solution \eqref{solution} of the transmission problem becomes
%\begin{align}
%u(z)=& (\kappa\alpha_2-\alpha_1)\sum_{m=1}^{\infty}\left[c_{-m}\left(s_m^{(1)}+\mathrm{i}s_m^{(2)}\right)+c_m\left(\frac{1}{\alpha_2} mA_m\right)\right]-J(z)\nonumber\\
%-&\alpha_1\sum_{m=1}^{\infty}\frac{1}{m}{\left(s_m^{(1)}+\mathrm{i}s_m^{(2)}\right)}\overline{F_m(z)}-\alpha_2\sum_{m=1}^{\infty}\overline{\left(s_m^{(1)}+\mathrm{i}s_m^{(2)}\right)}\sum_{k=m+1}^\infty \frac{{a_k}}{k-m}\overline{F_{k-m}'(z)}  .                               \label{solution_ellipse}
%\end{align}

%{\color{red}
%\begin{itemize}
%\item Uniqueness in the expansion
%\end{itemize}}
%
%{\color{red}Maybe add a note showing that when the far-field loading $u_0$ is linear, then $u$ is linear as well?}

In figure \ref{fig:ellipse}, we provide the graph of the density function as well as the level curves of the single layer potential when the far-field loading is linear. 
\begin{figure}[h!]
\begin{subfigure}[b]{0.5\textwidth}
	\centering
	\includegraphics[width=\linewidth]{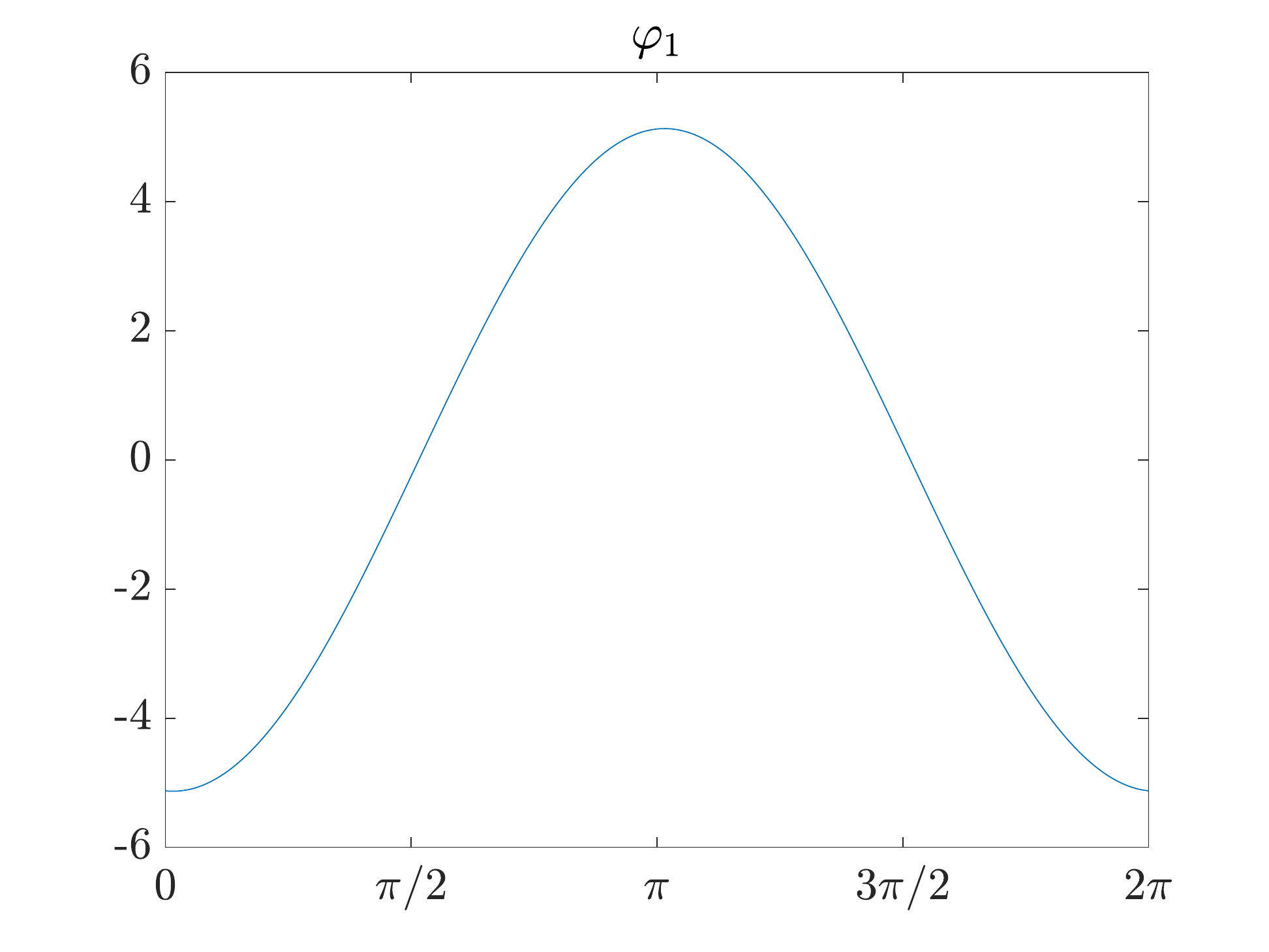}
	\caption{}
\end{subfigure}
\begin{subfigure}[b]{0.5\textwidth}
	\centering
	\includegraphics[width=\linewidth]{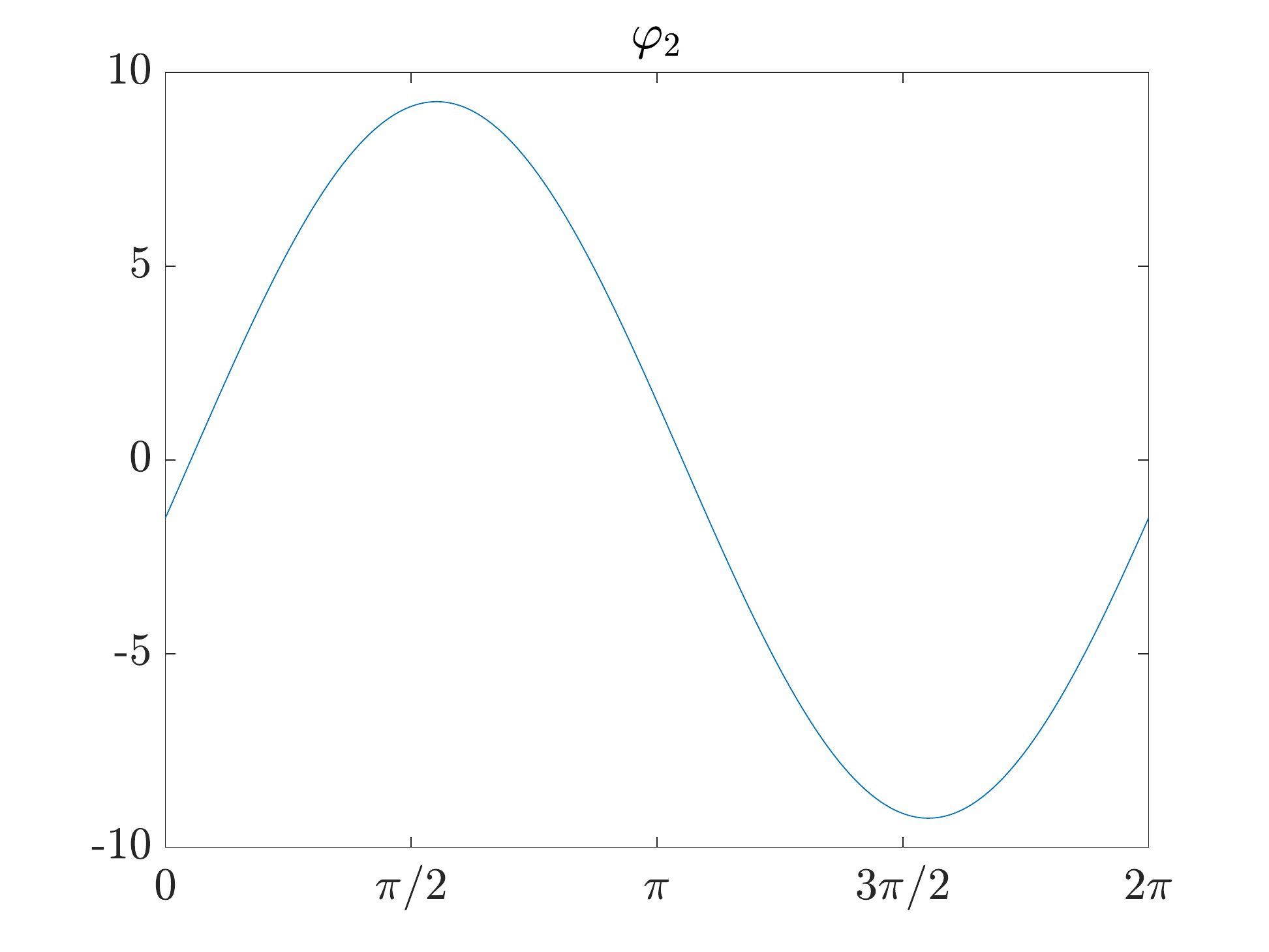}
	\caption{}
\end{subfigure}
\begin{subfigure}[b]{0.5\textwidth}
	\centering
	\includegraphics[width=\linewidth]{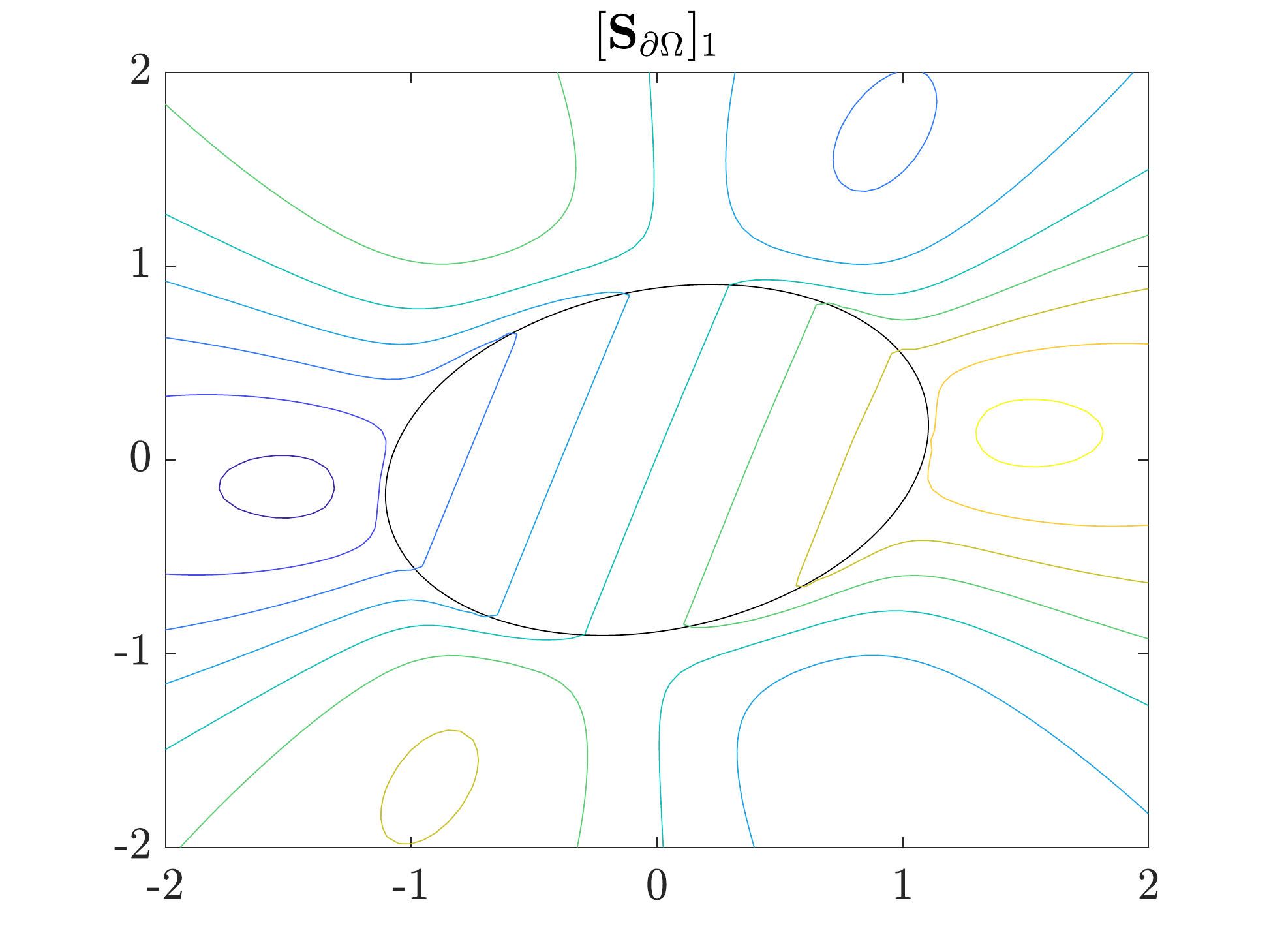}
	\caption{}
\end{subfigure}
\begin{subfigure}[b]{0.5\textwidth}
	\centering
	\includegraphics[width=\linewidth]{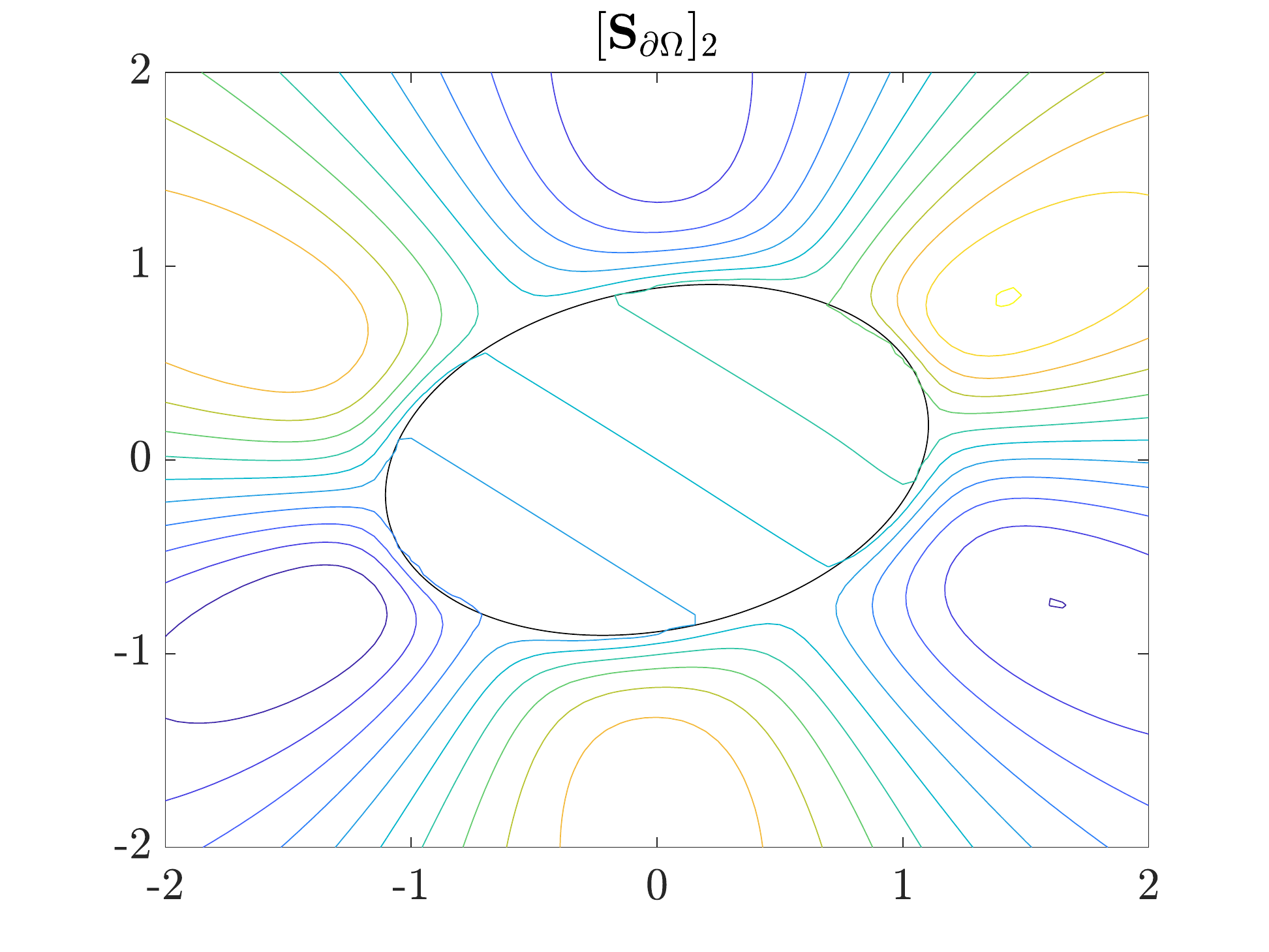}
	\caption{}
\end{subfigure}
%		\subfloat[]{\includegraphics[width=\textwidth]{bounds_v.png}}\\
%	\subfloat[]{\includegraphics[width=0.5\textwidth]{signal_bounds_Re_v.pdf}}
\caption{(a)--(b) components of the density function \(\boldsymbol{\varphi}\), (c)--(d) level curves of the components of the single layer potential \(\mathbf{S}_{\partial\Omega}[\boldsymbol{\varphi}]\), when \(\alpha_1=0.5\), \(\kappa=0.3\), and the inclusion is an ellipse with \(a_1=0.1+0.1\mathrm{i}\) subject to a linear far-field loading \(A_1=B_1=1\).}\label{fig:ellipse}
\end{figure}

\subsection{Inclusion of order 2 with an arbitrary far-field loading}

Let us consider an inclusion described by the following exterior conformal mapping:
\begin{equation*}
\Psi(w)=w+\frac{a_1}{w}+\frac{a_2}{w^2}.
\end{equation*}
In this case, it is not straightforward to determine an analytical expression for the Faber polynomials and their derivatives as it is for the ellipse case. Therefore, one has to use the recursive formula \eqref{eqn:Faberrecursion} to generate the Faber polynomials and, consequently, their derivatives.

Again, the coefficients $s_m^{(3)}$ and $s_m^{(4)}$, $m>1$,  are explicitly provided by \eqref{s_3_4}, whereas the coefficients $s_m^{(3)}$ and $s_m^{(4)}$ are given by \eqref{eqn:withJ}, in which 
{\beq\notag
\mathcal{P}\left[\mathbf{s} \right](z)=-\alpha_1\sum_{m=1}^{\infty}\frac{1}{m}{\left(s_m^{(1)}+\mathrm{i}s_m^{(2)}\right)}\overline{F_m(z)}+\mbox{Const.}
\eeq
and
\begin{align*}
J_1(z)&=-\mathrm{i}\frac{2}{\kappa}\left(\frac{a_1}{2}\overline{F_2'(z)}+\frac{a_2}{3}\overline{F_3'(z)}\right),\\
J_2(z)&=\sum_{m=1}^{\infty}\overline{B_m}\,\overline{F_m(z)}
+\overline{A_1}\left(\frac{a_1}{2}\overline{F_2'(z)}+\frac{a_2}{3}\overline{F_3'(z)}\right)
+\sum_{m=2}^{\infty}m\overline{A_m}\sum_{k=-1}^2 \frac{{a_k}}{k+m}\overline{F_{k+m}'(z)}.
%J_2(z)&=\sum_{m=1}^{\infty}\overline{B_m}\,\overline{F_m(z)}
%+\sum_{m=1}^{\infty}m\overline{A_m}\sum_{k=-1}^2 a_k\overline{\wtF_{k+m}(z)}.
\end{align*} 
}
{
	In matrix form, since $\overline{\boldsymbol{\Gamma}}^T\mathbf{D\,A}=\mathbf{0}$, the system \eqnref{matrix_block_system} turns into
	\beq\notag
	\left[\begin{array}{cc}
		\kappa\,\mathbf{D}&\mathbf{0}\\[3mm]
		\mathbf{0}&\kappa\,\mathbf{D}
	\end{array}
	\right]
	\left[
	\begin{array}{c}
		\mathbf{s}\\[3mm]
		\overline{\mathbf{s}}
	\end{array}
	\right]
	=\left[
	\begin{array}{c}
		\mathbf{y}\\[3mm]
		\overline{\mathbf{y}}
	\end{array}
	\right],
	\eeq
	where $\mathbf{y}$ is given by \eqnref{def:mathbfy}.
}
Again, thanks to the invertibility of the matrix $\mathbf{D}$, the constant $c_3$ and the coefficients  $s_m^{(1)}$ and $s_m^{(2)}$ can be easily found. 

In figure \ref{fig:2ndorder}, we provide the graph of the density function as well as the level curves of the single layer potential when the far-field loading is linear. 
\begin{figure}[h!]
	\begin{subfigure}[b]{0.5\textwidth}
		\centering
		\includegraphics[width=\linewidth]{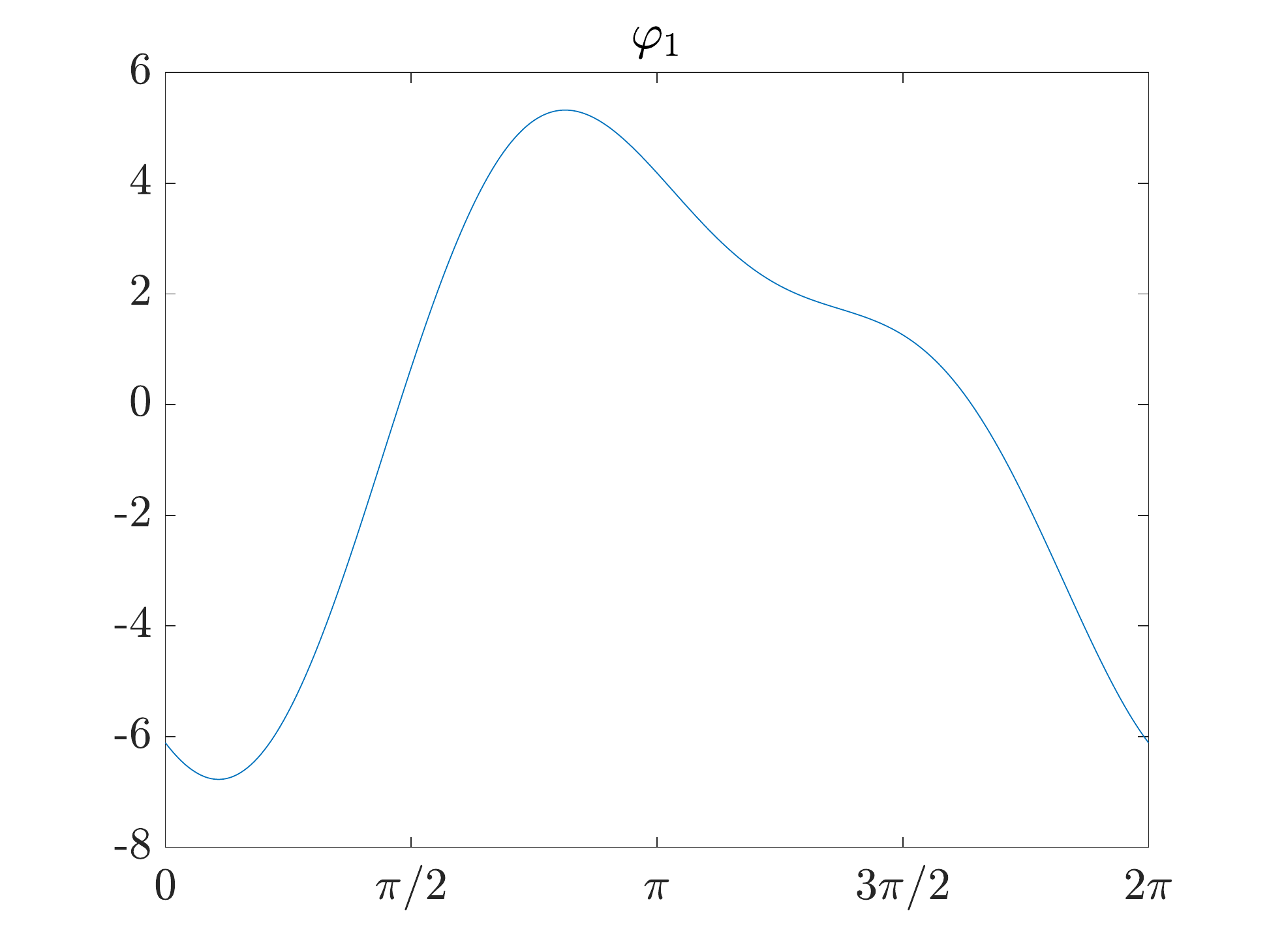}
		\caption{}
	\end{subfigure}
	\begin{subfigure}[b]{0.5\textwidth}
		\centering
		\includegraphics[width=\linewidth]{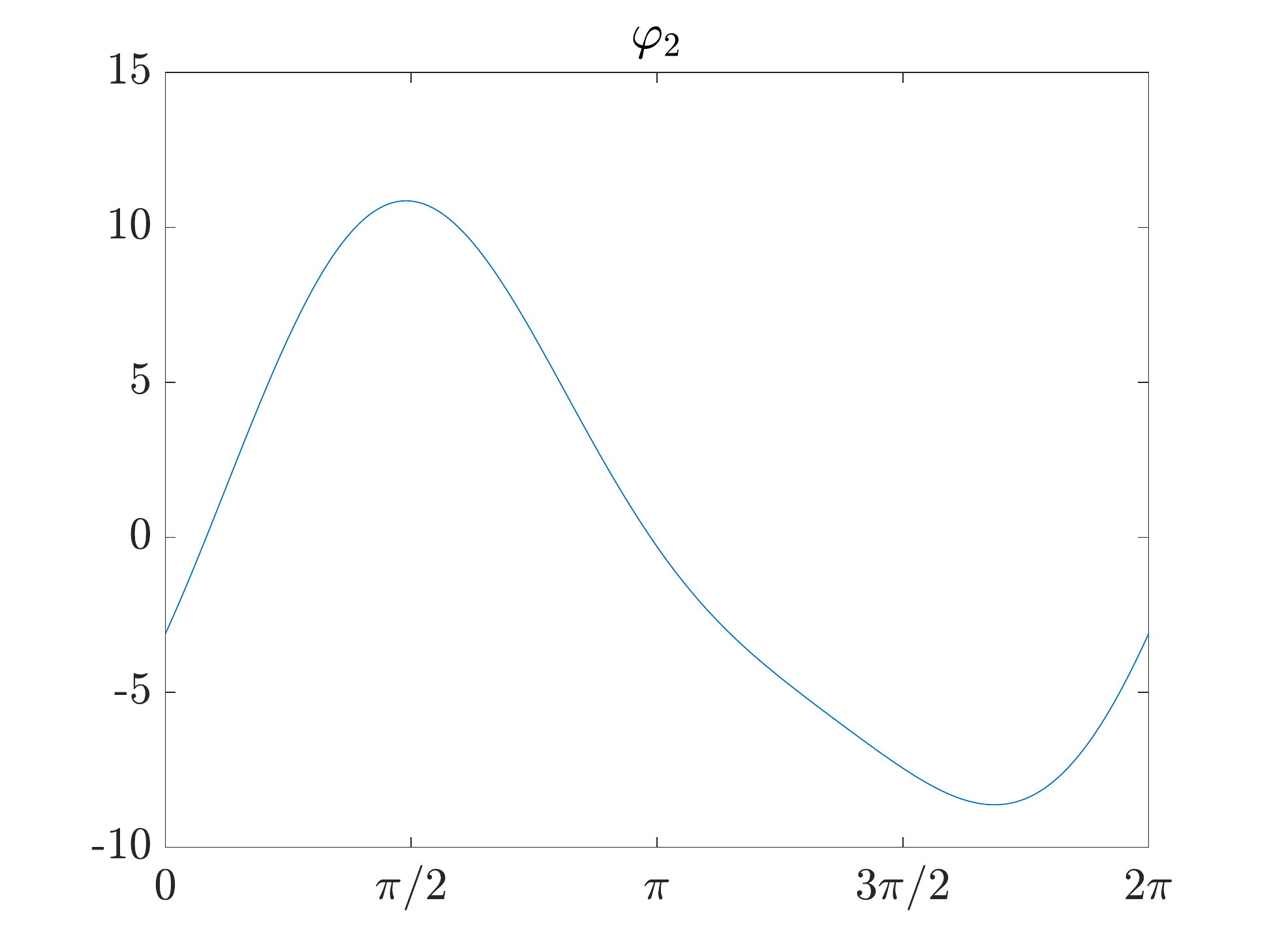}
		\caption{}
	\end{subfigure}
	\begin{subfigure}[b]{0.5\textwidth}
		\centering
		\includegraphics[width=\linewidth]{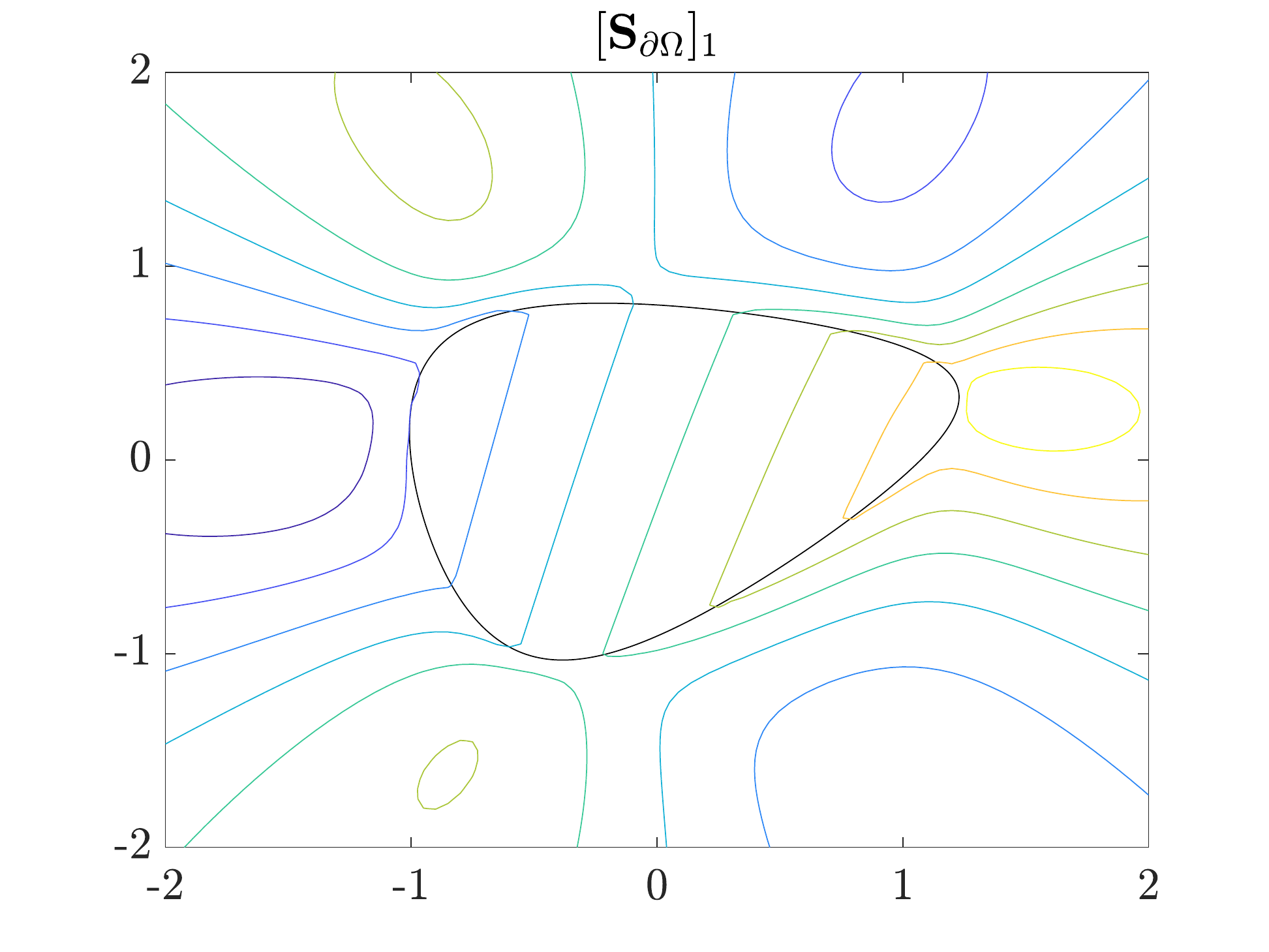}
		\caption{}
	\end{subfigure}
	\begin{subfigure}[b]{0.5\textwidth}
		\centering
		\includegraphics[width=\linewidth]{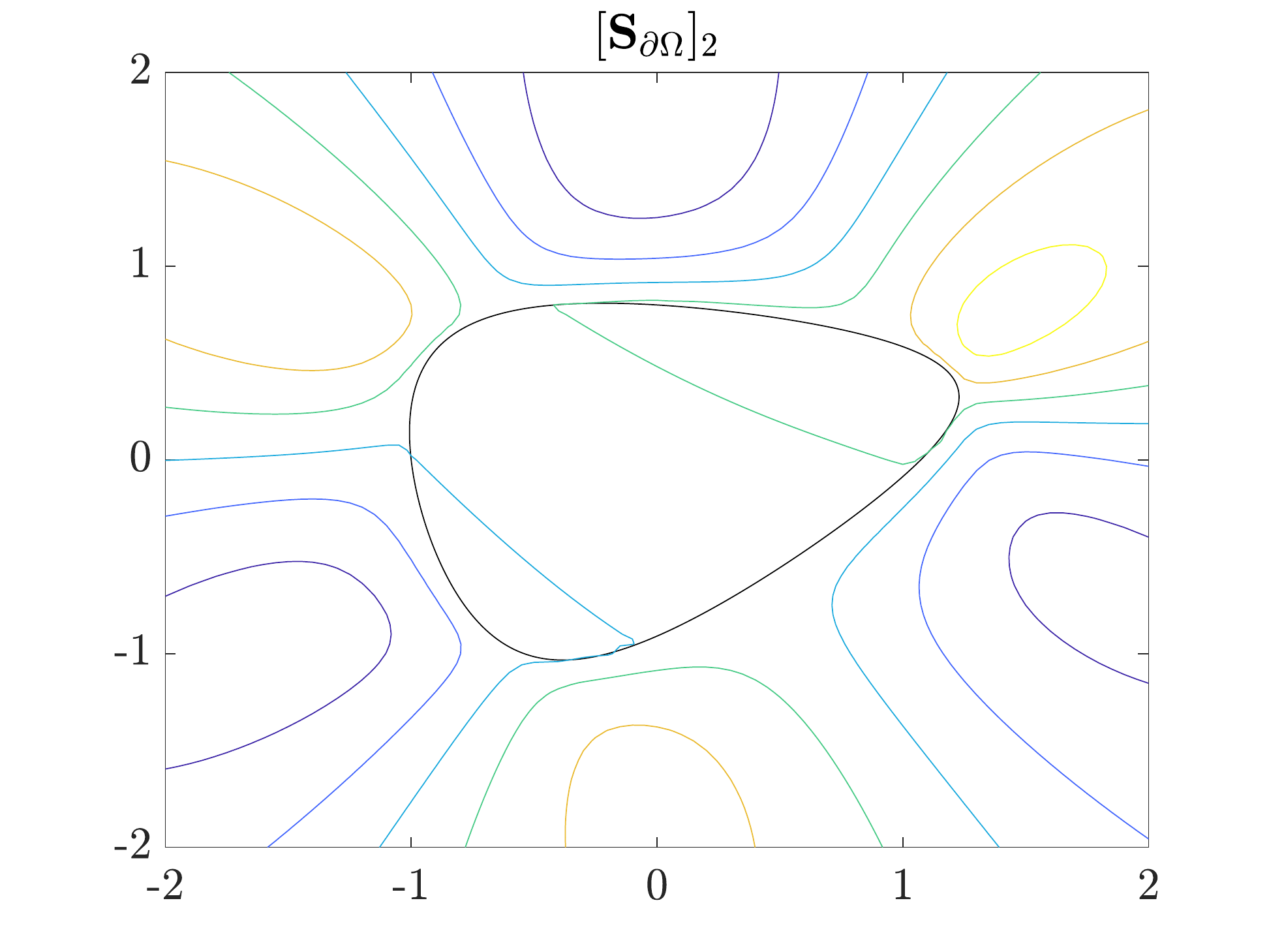}
		\caption{}
	\end{subfigure}
	%		\subfloat[]{\includegraphics[width=\textwidth]{bounds_v.png}}\\
	%	\subfloat[]{\includegraphics[width=0.5\textwidth]{signal_bounds_Re_v.pdf}}
	\caption{(a)--(b) components of the density function \(\boldsymbol{\varphi}\), (c)--(d) level curves of the components of the single layer potential \(\mathbf{S}_{\partial\Omega}[\boldsymbol{\varphi}]\), when \(\alpha_1=0.5\), \(\kappa=0.3\), and the inclusion is a second-order inclusion with \(a_1=a_2=0.1+0.1\mathrm{i}\) subject to a linear far-field loading \(A_1=B_1=1\).}\label{fig:2ndorder}
\end{figure}

\subsection{Inclusion of order 3 with an arbitrary far-field loading}

Let us consider now the case of an inclusion of order 3, for which the conformal mapping looks like
\begin{equation*}
\Psi(w)=w+\frac{a_1}{w}+\frac{a_2}{w^2}+\frac{a_3}{w^3}.
\end{equation*}
Again, the coefficients $s_m^{(3)}$ and $s_m^{(4)}$ are given by \eqref{s_3_4}, whereas the coefficients $s_m^{(1)}$ and $s_m^{(2)}$ are provided by \eqref{eqn:withJ} in which 
{
	\begin{equation}\label{sol_order3}
	\mathcal{P}[\mathbf{s}](z)=-\alpha_1\sum_{m=1}^{\infty}\frac{1}{m}{\left(s_m^{(1)}+\mathrm{i}s_m^{(2)}\right)}\overline{F_m(z)}-\alpha_2 a_3\overline{{\left(s_2^{(1)}+\mathrm{i}s_2^{(2)}\right)}}\,\overline{F_{1}(z)}
	+\mbox{Const.}
	\end{equation}
	and
	\begin{align*}
	J_1(z)&=- \mathrm{i}\frac{2}{\kappa}\sum_{k=1}^3 \frac{{a_k}}{1+k} \overline{F_{1+k}'(z)},\notag\\
	J_2(z)&= \sum_{m=1}^{\infty}\overline{B_m}\,\overline{F_m(z)}
	+\overline{A_1}\sum_{k=1}^3\frac{{a_k}}{k+1} \overline{F_{k+1}'(z)}
	+\sum_{m=2}^{\infty}m\overline{A_m}\sum_{k=-1}^3  \frac{{a_k}}{k+m}\overline{F_{k+m}'(z)}.
	\end{align*}
}
In this case, one has to invert the block matrix in \eqref{matrix_block_system}, which is invertible, given that the matrix $\overline{\boldsymbol{\Gamma}}^T\mathbf{A}\mathbf{D}$ has only one non-zero entry, that is, {the $(1,1)$-entry}. 

In figure \ref{fig:3ndorder}, we provide the graph of the density function as well as the level curves of the single layer potential when the far-field loading is linear. 
\begin{figure}[h!]
	\begin{subfigure}[b]{0.5\textwidth}
		\centering
		\includegraphics[width=\linewidth]{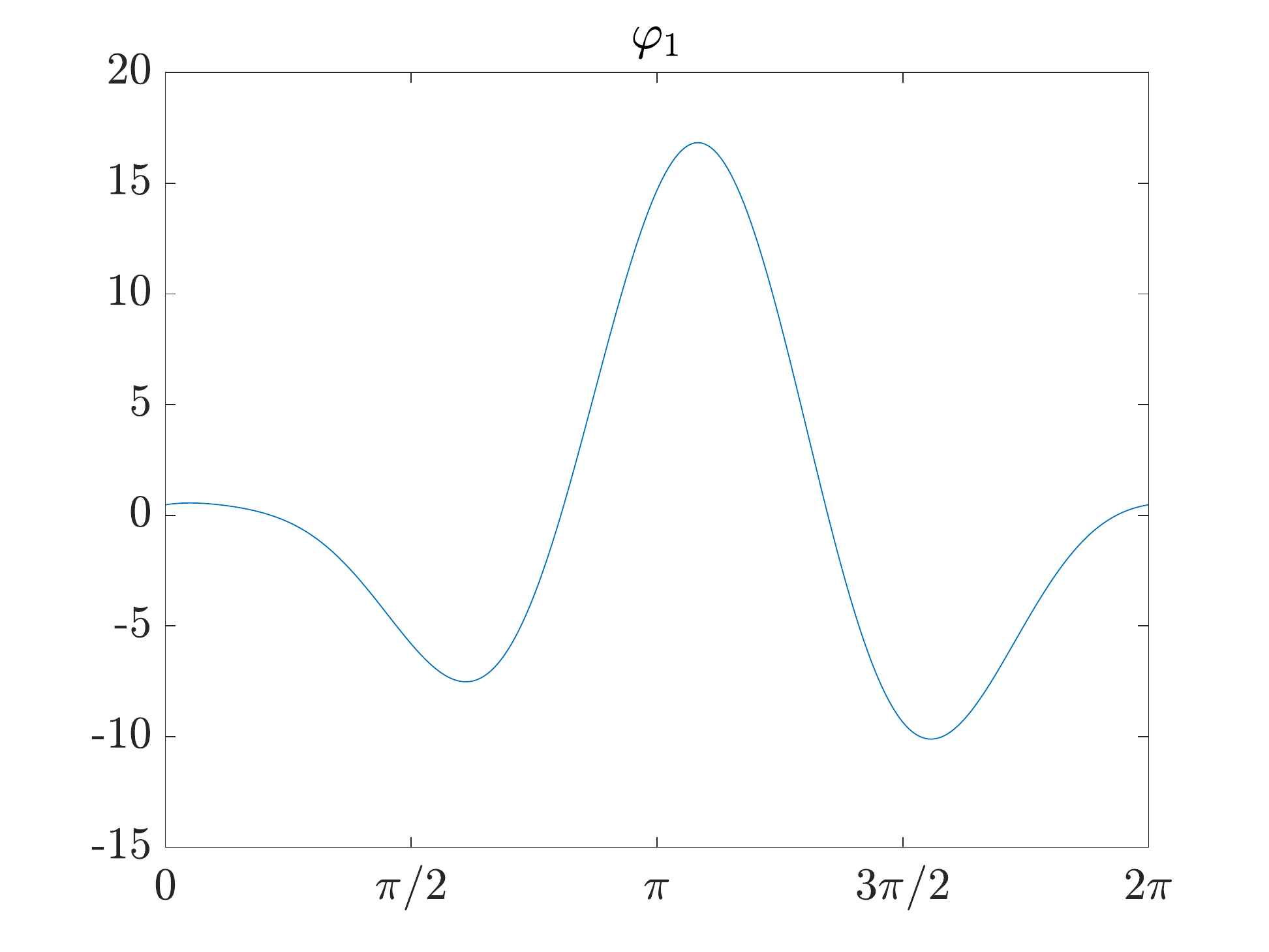}
		\caption{}
	\end{subfigure}
	\begin{subfigure}[b]{0.5\textwidth}
		\centering
		\includegraphics[width=\linewidth]{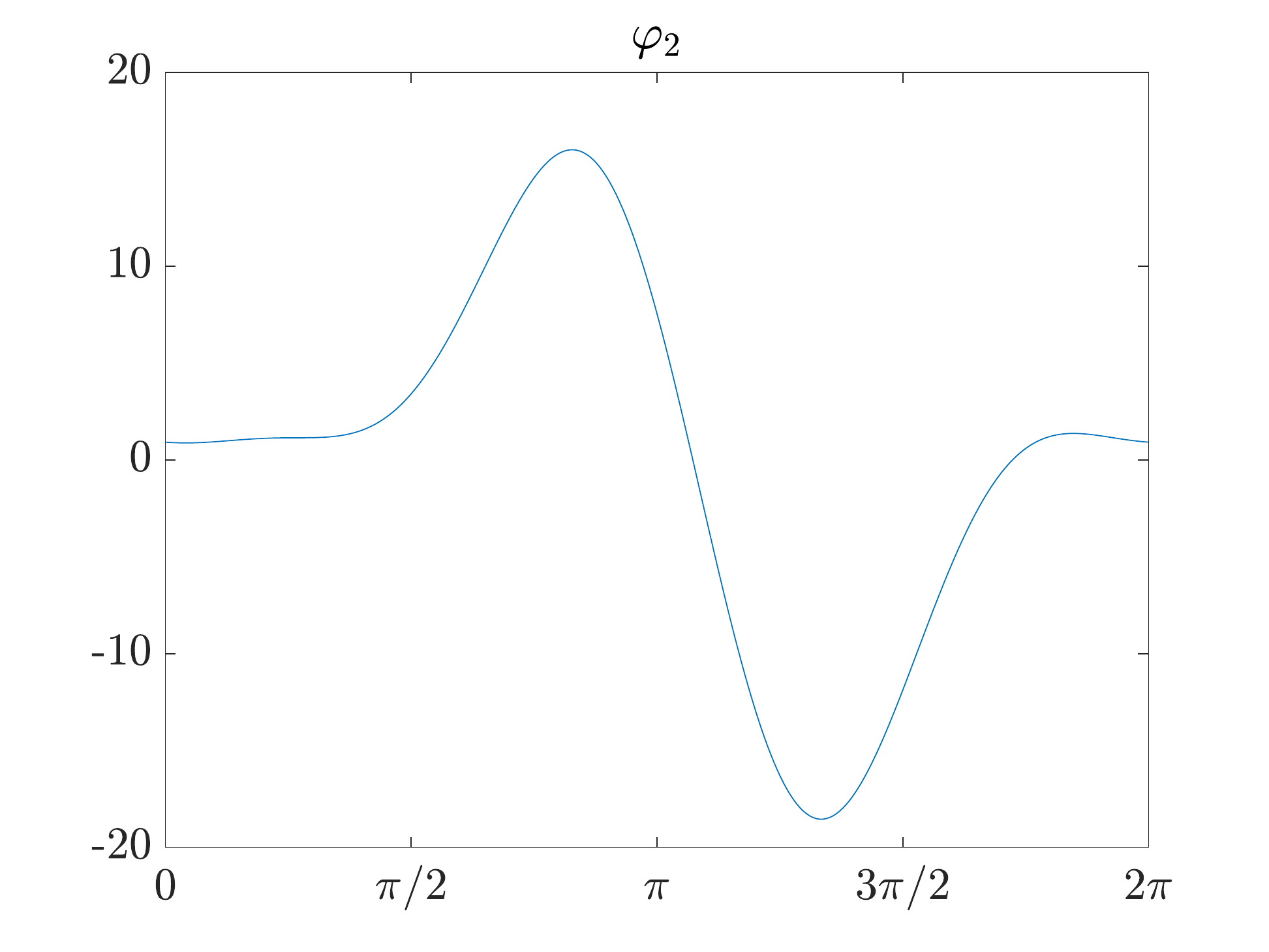}
		\caption{}
	\end{subfigure}
	\begin{subfigure}[b]{0.5\textwidth}
		\centering
		\includegraphics[width=\linewidth]{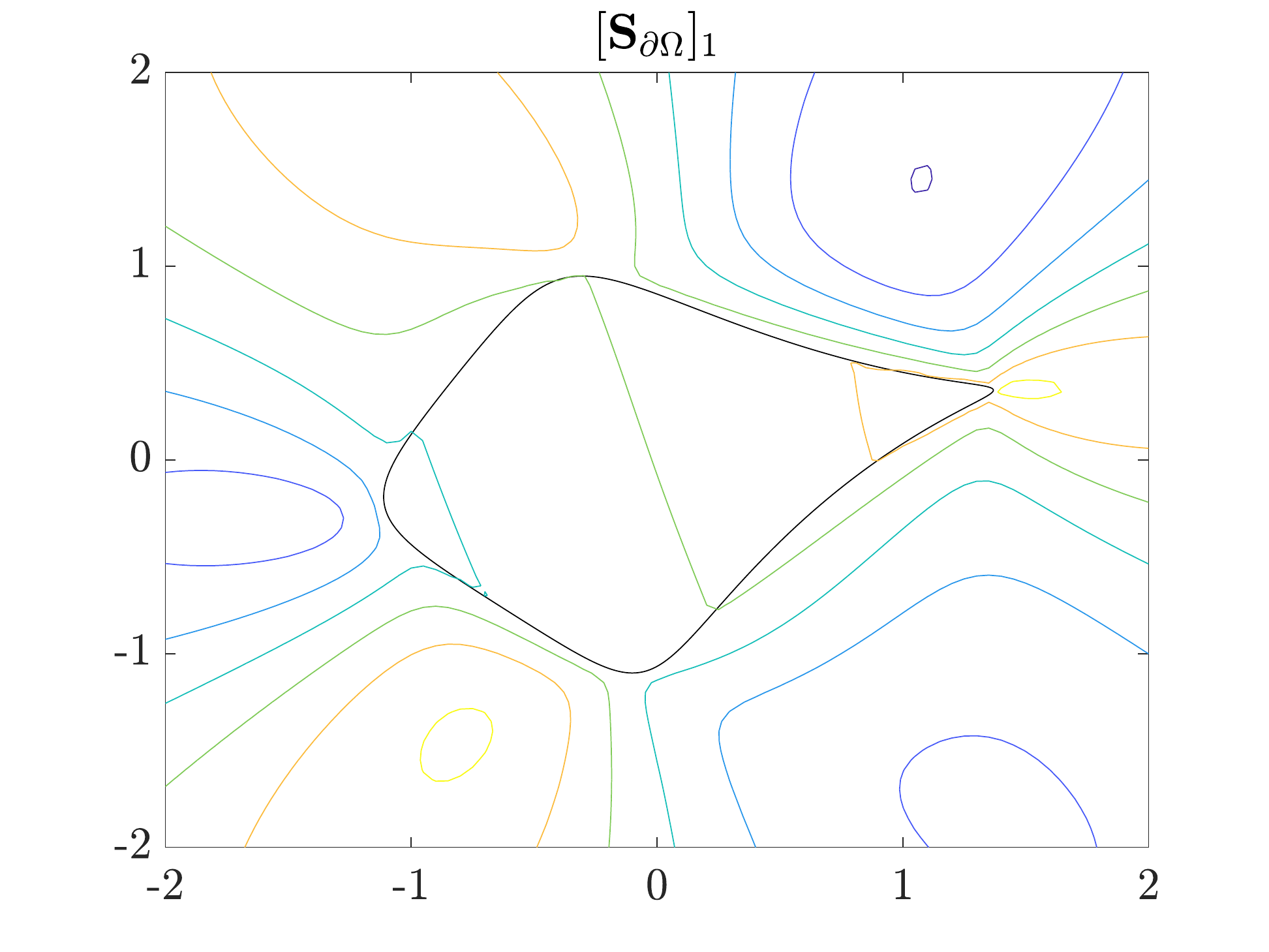}
		\caption{}
	\end{subfigure}
	\begin{subfigure}[b]{0.5\textwidth}
		\centering
		\includegraphics[width=\linewidth]{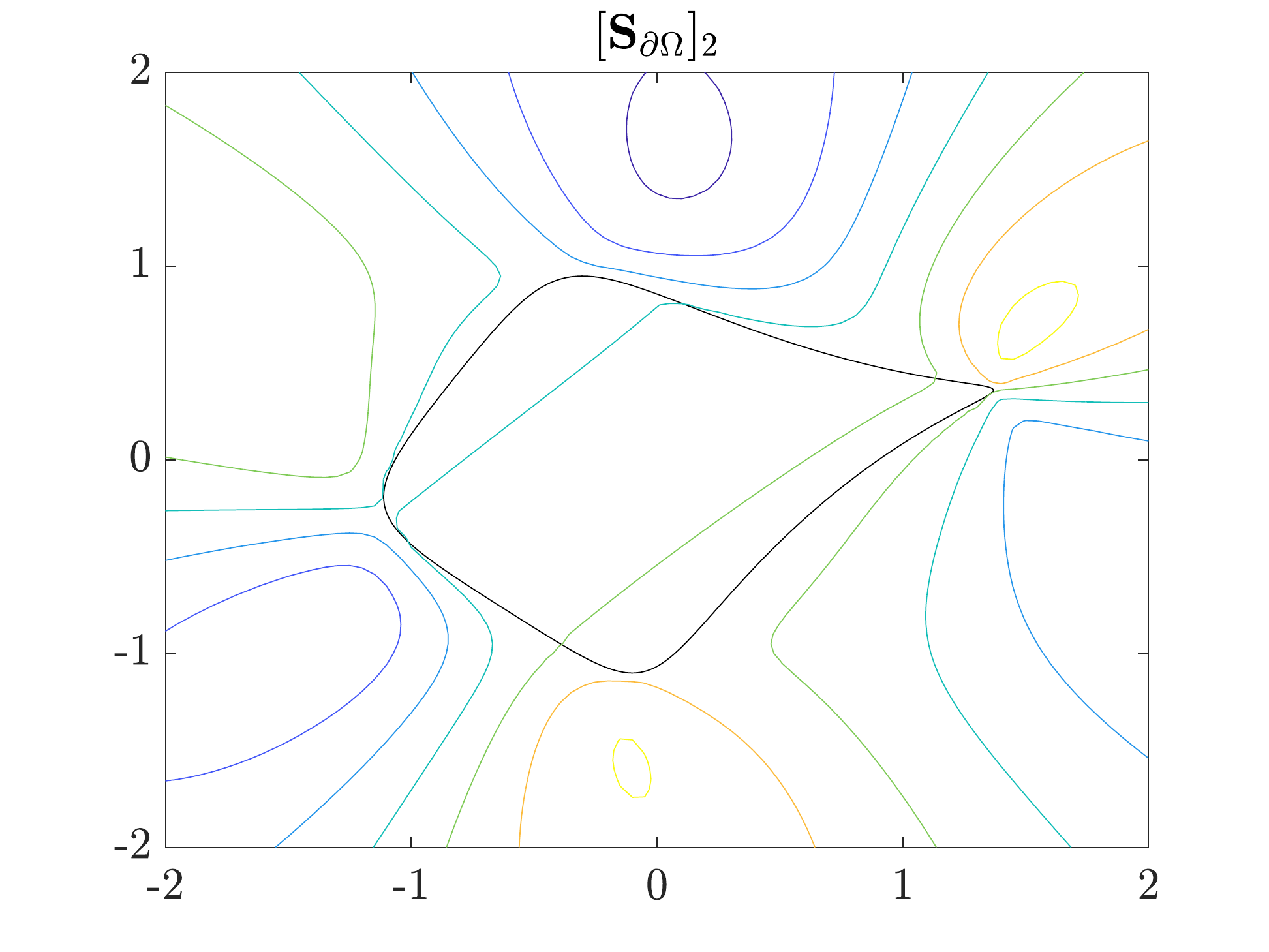}
		\caption{}
	\end{subfigure}
	%		\subfloat[]{\includegraphics[width=\textwidth]{bounds_v.png}}\\
	%	\subfloat[]{\includegraphics[width=0.5\textwidth]{signal_bounds_Re_v.pdf}}
	\caption{(a)--(b) components of the density function \(\boldsymbol{\varphi}\), (c)--(d) level curves of the components of the single layer potential \(\mathbf{S}_{\partial\Omega}[\boldsymbol{\varphi}]\), when \(\alpha_1=0.5\), \(\kappa=0.3\), and the inclusion is a third-order inclusion with \(a_1=a_2=a_3=0.1+0.1\mathrm{i}\) subject to a linear far-field loading \(A_1=B_1=1\).}\label{fig:3ndorder}
\end{figure}

\subsection{Matrix formulation for inclusions of higher order with an arbitrary far-field loading}

{
	Suppose now that the domain $\Omega$ has order $M$, that is, the corresponding conformal mapping \eqref{eqn:extmapping} is
	\begin{equation*}
	\Psi(w)=w+a_0+\frac{a_1}{w}+\cdots+\frac{a_M}{w^M}\ \left(a_M\neq0\right).
	\end{equation*}
	Then, the matrix $\mathbf{A}$ defined by \eqref{b_A_D_def} is an upper anti-triangular matrix
	and so is $\overline{\boldsymbol{\Gamma}^T}\mathbf{D\,A}$, given that $\mathbf{D}$ is diagonal (see \eqref{b_A_D_def}) and $\overline{\boldsymbol{\Gamma}^T}$ is upper triangular (see \eqref{Gamma_F}).
	{In such a case, the $(i,j)$-th entry of $\overline{\boldsymbol{\Gamma}^T}\mathbf{D\,{A}}$ is zero for all $i+j\geq M$ so that Problem (A) (or (A$^\prime$)) can be reduced to the following.}
	\begin{problem*}[A$^{\prime \prime}$] Find $\mathbf{s}_{M-2}=(s_m^{(1)}+\mathrm{i}s_m^{(2)})_{1\leq m\leq M-2}$
		satisfying
		\beq\label{mat:eqn:M}
		\left[\begin{array}{cc}
			\kappa\mathbf{D}_{M-2}	&  \left(\overline{\boldsymbol{\Gamma}^T}\mathbf{D\,A}\right)_{M-2} \\[3mm]
			\left(\boldsymbol{\Gamma}^T\mathbf{D}\,\overline{\mathbf{A}}\right)_{M-2} &	\kappa\mathbf{D}_{M-2}
		\end{array}
		\right]
		\left[
		\begin{array}{c}
			\mathbf{s}_{M-2}\\[3mm]
			\overline{\mathbf{s}}_{M-2}
		\end{array}
		\right]
		=\left[
		\begin{array}{c}
			\mathbf{y}_{M-2}\\[3mm]
			\overline{\mathbf{y}}_{M-2}
		\end{array}
		\right]
		\eeq
		and $\,\widehat{\mathbf{s}}_{M-2}=(s_m^{(1)}+\mathrm{i}s_m^{(2)})_{m> M-2}$ satisfying 
		\beq\label{mat:finiteM:higher}
		\left[\begin{array}{cc}
			\kappa\,\widehat{\mathbf{D}}_{M-2}&\mathbf{0}\\[3mm]
			\mathbf{0}&\kappa\,\widehat{\mathbf{D}}_{M-2}
		\end{array}
		\right]
		\left[
		\begin{array}{c}
			\widehat{\mathbf{s}}_{M-2}\\[3mm]
			\overline{\widehat{\mathbf{s}}}_{M-2}
		\end{array}
		\right]
		=\left[
		\begin{array}{c}
			\widehat{\mathbf{y}}_{M-2}\\[3mm]
			\overline{\widehat{\mathbf{y}}}_{M-2}
		\end{array}
		\right].
		\eeq
		Here, for notational convenience, we denote by $\mathbf{v}_{n}=(v_i)_{i=1}^n$ the $n\times 1$ subvector of a given vector $\mathbf{v}=(v_i)_{i=1}^\infty$ and $\mathbf{B}_{n}$ the $n\times n$ submatrix $(b_{ij})_{i,j=1}^n$ of a given matrix $\mathbf{B}=(b_{ij})_{i,j=1}^\infty$. We also set $\widehat{\mathbf{v}}_n=(v_i)_{i= n+1}^\infty$ and $\widehat{\mathbf{B}}_n=(b_{ij})_{i,j=n+1}^\infty$.
	\end{problem*} 
	As a consequence of \eqnref{mat:finiteM:higher}, the coefficients  $s_m^{(1)}+\mathrm{i}s_m^{(2)}$ with $m> M-2$ can be determined in a unique way:
	{\beq\label{sol:orderM}
	s_m^{(1)}+\mathrm{i}s_m^{(2)}= \frac{m}{\kappa}\,\overline{y}_m\quad\mbox{for }m> M-2.\eeq
	}
	Clearly, to determine the coefficients $s_m^{(1)}+\mathrm{i}s_m^{(2)}$ with $m\leq M-2$, one has first to assess the invertibility of the $(2M-4)\times(2M-4)$ matrix in \eqnref{mat:eqn:M}.
	
}

\section{Conclusion}

The plane elastostatic transmission problem is a classical problem in Applied Mechanics, for which the solution is provided explicitly only when the inclusion has a simple shape, such as in the case of an ellipsoidal inclusion for which one can use elliptic coordinates (see, e.g., \cite{Ando:2018:SPN}). For an inclusion of arbitrary shape such a coordinate system can be defined only locally, thus preventing one from finding an explicit series solution to the  transmission problem.
To the best of our knowledge, there has been no previous work that provides an explicit series solution in the case of a domain of arbitrary shape with an arbitrary far-field loading.

The key idea of our approach, successfully applied to the conductivity case in \cite{Jung:2018:SSM}, consists in the introduction of two sets of bases, one for analytic functions defined outside of the inclusion and one for analytic functions defined inside the inclusion. The exterior basis is based on the coordinate system introduced by the external conformal mapping associated with the inclusion, whereas the interior basis is based on the Faber polynomials associated with the inclusion. The introduction of Faber polynomials 
allows one to overcome the drawbacks associated with the use of conformal mappings, and it presents
a novel method to determine the solution of the elastic transmission problem in an elegant way. Indeed, thanks to the properties of Faber polynomials, the transmission condition at the boundary of the inclusion allows one to derive an explicit formula for the coefficients of the series expansion of the transmission problem in terms of the coefficients of the series expansion of the far-field loading, supposed to be arbitrary. Specifically, in this work, we provide an explicit analytical formula in the case of an arbitrary far-field loading and an algebraic inclusion of order up to 3, whereas, for higher order domains, some numerical computations are required.

This paper represents the first step towards a complete characterization of the plane elastostatic transmission problem.  Indeed, for the general case, besides the transmission condition regarding the displacements, one should also consider the one concerning the continuity of tractions at the boundary of the inclusion. This would put forward another research avenue{\textemdash}the solution of the so-called \textit{E-inclusion problem} for the plane elastostatic case{\textemdash}thus extending the results found in \cite{Kim:2018:EEC:preprint} for the conductivity problem. Such a  problem involves the determination of  the shape of the inclusion that provides uniform fields inside the inclusion for any or some applied far-field loadings. Note that finding E-inclusions is an important problem in many practical applications concerning the design of materials which induce stress fields with small variances in the inclusion phase: these inclusions, which are tailored to the applied field, are generally less likely to break down than inclusions with large variances of the stress field.
The ultimate goal would be to solve the elastostatic \textit{neutral inclusion problem}: some coated inclusions, when placed in a medium, do not disturb the exterior field, and these are denoted as neutral inclusions. Once a neutral inclusion has been found, similar inclusions, possibly of different sizes, can be added to the background matrix without altering the exterior uniform field (e.g., \cite{Hashin:1962:EMH}). In this way it becomes possible to construct a composite, consisting of multiple inclusions and a background matrix, whose effective property exactly coincides with that of the matrix.

\section*{Acknowledgments}
{ML is supported by the Basic Science Research Program through the National Research Foundation of Korea (NRF) funded by the Ministry of Education (2019R1A6A1A10073887). OM is grateful to the National Science Foundation for support through the Research Grant DMS-942514985.

 \section*{Conflict of interest}

 The authors declare that they have no conflict of interest.

%\bibliography{tcbook,newref, 2020_Elastic_series_bib}

%\begin{thebibliography}{}
%%
%% and use \bibitem to create references. Consult the Instructions
%% for authors for reference list style.
%\end{thebibliography}

\end{document}